\documentclass{amsart}

\usepackage{parskip}

\usepackage[utf8]{inputenc}
\usepackage[english]{babel}
\usepackage[T1,T2A]{fontenc}

\usepackage[autostyle]{csquotes}

\usepackage[
backend=biber,
style=numeric,
sorting=nyt,
natbib=true,
maxbibnames=99,
url=true, 
doi=true,
eprint=false
]{biblatex}
\usepackage[]{hyperref}
\hypersetup{
	colorlinks=true,
}

\renewbibmacro{in:}{%
	\ifentrytype{article}{}{%
		\printtext{\bibstring{in}\intitlepunct}}}

\usepackage{epigraph}
\usepackage{amsmath}
\usepackage{amssymb}
\usepackage{bm}
\usepackage{amsthm}
\usepackage{amsfonts}
\usepackage{mathtools}
\usepackage{parskip}
\usepackage{enumitem}
\usepackage{physics}
\usepackage{pstricks}
\usepackage{multido}
\usepackage{graphicx}
\graphicspath{ {./images/} }
\usepackage{caption}
\usepackage{subcaption}
\usepackage{tikz-cd}
\usepackage{tikz}
\usetikzlibrary{positioning}
\tikzset{mynode/.style={draw,circle,inner sep=2pt,outer sep=0pt}}
\usetikzlibrary{intersections}
\usetikzlibrary{decorations.pathmorphing}

\newtheorem{thm}{Theorem}[section]
\newtheorem{cor}[thm]{Corollary}
\newtheorem{lem}[thm]{Lemma}
\newtheorem{prop}[thm]{Proposition}

\newtheorem*{thm*}{Theorem}
\newtheorem*{conj*}{Conjecture}

\theoremstyle{definition}
\newtheorem{defn}[thm]{Definition}

\newcommand*{\ve}{\varepsilon}

\newcommand*{\ga}{\ga}

\newcommand{\defeq}{\vcentcolon=}
\newcommand{\overbar}[1]{\mkern 1.5mu\overline{\mkern-1.5mu#1\mkern-1.5mu}\mkern 1.5mu}

\newcommand{\ltens}{
	\otimes^{\textbf{L}}
}

\newcommand{\rhom}{
	\textbf{R}\Hom
}

\newcommand{\opcat}[1]{
	{#1}^{\mathrm{op}}
}

\DeclareMathOperator{\rad}{rad}
\DeclareMathOperator{\soc}{soc}

\DeclareMathOperator{\cone}{cone}

\DeclareMathOperator{\id}{id}

\DeclareMathOperator{\Ext}{Ext}

\DeclareMathOperator{\End}{End}
\DeclareMathOperator{\Hom}{Hom}

\DeclareMathOperator{\Id}{Id}

\DeclareMathOperator{\Perf}{Perf}
\DeclareMathOperator{\stmod}{-\underline{mod}}
\DeclareMathOperator{\Ch}{Ch}

\DeclareMathOperator{\modcat}{-mod}

\DeclareMathOperator{\cadd}{-add}
\DeclareMathOperator{\proj}{-proj}

\DeclareMathOperator{\SL}{SL}

\newcommand{\bbimder}[2]{
	D^b({#1}\text{-}{#2})
}

\newcommand{\bbimcomp}[2]{
	\Ch^b({#1}\text{-}{#2})
}

\title{Strongly Periodic Modules and Perverse Autoequivalences}
\author{Alfred Dabson}
\date{}

\begin{document}
	
	\maketitle
	
	\begin{abstract}
		We introduce a notion of \emph{strong periodicity} of a module over a finite-dimensional algebra over a field. We prove that the existence of such modules over certain idempotent algebras is both a necessary and sufficient condition for the existence of a two-step self-perverse equivalence of a finite-dimensional algebra. We survey some applications to the setting of the symmetric groups.
	\end{abstract}

	\section{Introduction}
	
	Perverse equivalences, as introduced by Chuang and Rouquier in \cite{chuang-rouquier_perverse17}, are a class of derived equivalences, filtered by shifted Morita equivalences. This built on work in \cite{chuang-rouquier_sl2cat}, in which the authors defined derived equivalences arising from $\mathfrak{sl}_2$-categorification, equivalences that are themselves perverse. Chuang and Rouquier's work resolved Brou\'{e}'s Abelian Defect Group Conjecture \cite{broue_conjecture} for blocks of symmetric groups, highlighting the significance of perverse equivalences in the representation theory of finite groups. This has been explored further by Craven and Rouquier, \cite{craven_perverse} and \cite{craven-rouquier_perverse}, proving for some blocks of sporadic groups and finite groups of Lie type in non-defining characteristic that conjectured derived equivalences realising Brou\'{e}'s conjecture coming from Deligne-Lusztig varieties exist and are perverse. Independently, Wong \cite{wong_SL2defining} has exploited perverse equivalences to resolve Brou\'{e}'s conjecture for blocks of $\SL_2(q)$ in defining characteristic in a novel way.
	
	The significance of perverse equivalences is by no means limited to the representation theory of finite groups. As the name might suggest, Chuang and Rouquier's constructions are are intimately related to the gluing of perverse sheaves in geometry \cite{BBD_perverse}, and in certain cases, \cite{anno-bezrukavnikov-mirkovic}, can be described in terms of Bridgeland stability conditions \cite{bridgeland-stability}, though this aspect is as of yet poorly understood.
	
	Perverse equivalences have a natural combinatorial flavour, making them far easier to work with than general derived equivalences. Moreover, unlike a general derived equivalence, a perverse equivalence between two algebras defines a stratified bijection between their sets of simple modules.
	
	The class of two-step perverse equivalences, for which filtration is of length two, can be considered the closest derived equivalences to a Morita equivalence. We call the magnitude $d$ of the singular shift the \emph{width} of the equivalence. Two-step perverse equivalences of width 1 are combinatorial or standard Okuyama-Rickard tilts \cite{okuyama_method}, examples of which include Kauer moves on Brauer graph algebras \cite{kauer_brauergraphs}. Iterated applications of combinatorial tilts produce two-step equivalences of larger width; it remains unknown if other sources of two-step perverse equivalences of width $d$ exist.
	
	Among the more compelling studies of two-step perverse equivalences comes from Grant \cite{grant_periodic}, \cite{grant_thesis}, who proves that, for symmetric algebras, periodic and relatively periodic idempotent algebras give rise to two-step perverse autoequivalences of width equal to the period of the idempotent algebra. In particular, if $k$ is a field and $A$ is a finite-dimensional symmetric $k$-algebra, $P$ a projective $A$-module, $E=\opcat{\End_A(P)}$ and $\sigma$ an automorphism of $E$, then if $E$ is $\sigma$-periodic of period $n$ relative to some symmetric subalgebra $B$ of $E$ and $P$ is a projective right $B$-module, then there is a two-step perverse autoequivalence \begin{tikzcd}[cramped,sep=small]
		\Psi_P: D^b(A) \arrow[r,"\sim"] & D^b(A)
	\end{tikzcd} of width $n$. Moreover, $\Psi_P$ coincides with the $n$th iterated application of a combinatorial tilt of $A$, producing a circle of two-step perverse equivalences of width 1, starting and ending at $A$. This provides a concrete link to geometric phenomena, specifically the spherical and $\mathbb{P}^n$-twists of Seidel and Thomas \cite{seidel-thomas_spherical} and Huybrechts and Thomas \cite{huybrechts-thomas_Pn} respectively.
	
	The work in this document was born out of a desire to better understand the relationship between periodicity and perversity, and in particular to strengthen Grant's result into an if-and-only if statement. 
	
	In Sections \ref{sec:background} and \ref{sec:periodicmodulesalgebras}, we present the relevant background, recounting Chuang and Rouquier's definition of perverse equivalences between finite-dimensional algebras and some significant first consequences, before a brief dip into periodicity, primarily to describe more fully Grant's result.
	
	Section \ref{sec:main} provides the mathematical core of this document. We introduce a notion of \emph{strong periodicity} of a module, by which an $E$-module $M$ is strongly $\sigma$-periodic of period $n$, for $\sigma$ an automorphism of $E$ and $n \in \mathbb{Z}_+$, if there is an $\alpha \in \Ext^n_{E \otimes_k \opcat{E}}(E,{}_{\sigma\!}E)$ such that $\alpha \ltens_E M$ induces an isomorphism $\Omega_E(M) \cong {}_{\sigma\!}M$. We have our first key result, Theorem \ref{thm:perversetoperiodic}.
	
	\begin{thm*}
		If $A$ is a finite-dimensional, symmetric $k$-algebra and \begin{tikzcd}[cramped,sep=small]
			\Phi:D^b(A) \arrow[r,"\sim"] & D^b(A)
		\end{tikzcd} is a two-step perverse equivalence of width $n$ satisfying a natural restriction condition, then there are projective $A$-modules $P$ and $Q$ such that, with $E=\opcat{\End_A(P)}$, the $E$-module $M=\Hom_A(P,Q)$ is strongly $\sigma$-periodic and the $\opcat{E}$-module $M^\vee$ is strongly $\sigma^{-1}$-periodic, both of period $n$, relative to some $\alpha\in \Ext^n_{E \otimes_k \opcat{E}}(E,{}_{\sigma\!}E)$.
	\end{thm*}	
	
	Working more directly within the derived category allows us to prove a converse, encompassing Grant's earlier result. Given a basic algebra $A$ and projective $A$-modules $P$ and $Q$ such that $A \cong P \oplus Q$ as $A$-modules, set $E=\opcat{\End_A(P)}$ and $M=\Hom_A(P,Q)$. Suppose that $M$ is a strongly periodic $E$-module of period $n$ and $M^\vee$ is a strongly periodic $\opcat{E}$-module of period $n$, with the periodicity arising from the same $\alpha \in \Ext^n_{E \otimes_k \opcat{E}}(E,{}_{\sigma\!}E)$. We construct an endofunctor \begin{tikzcd}[cramped, sep=small]
		\Phi_P: D^b(A) \arrow[r] & D^b(A)
	\end{tikzcd} from $\alpha$, called the \emph{generalised periodic twist} at $P$. We have the following, Theorem \ref{thm:periodictoperverse}.
	
	\begin{thm*}
		The generalised periodic twist $\Phi_P$ is a two-step perverse equivalence of width $n$.
	\end{thm*}
	
	We then demonstrate that, as in Grant's case, the equivalence $\Phi_P$ can be realised as the inverse of an $n$th iterated combinatorial tilt, producing a cycle of length $n$ of perverse equivalences between different algebras. One can follow this cycle from any starting point to produce a two-step perverse autoequivalence of width $n$, corresponding to our construction.
	
	Finally, in Section \ref{sec:symmgroups}, we look at two applications of our result to blocks of the symmetric groups, one well-known example occurring in the principal block of $\mathfrak{S}_6$ in characteristic 3, and one more surprising example in the principal block of $\mathfrak{S}_8$ in characteristic 3. This class of algebra, and the broader class of blocks of Iwahori-Hecke algebras of type $A$, appears to be a fruitful source of interesting equivalences, which may be worth exploring further.
	
	Our methods establish a firm connection between perverse equivalences and the (twisted) Hochschild cohomology of idempotent algebras. Work of Keller \cite{keller_DGHochschildinv}, \cite{keller_DGHochschildinv} indicates the possibility of extending these results to the differential graded setting, and that this may in fact be a more natural point of view to take. We leave this as tantalising potential future work.

	\section{Background}\label{sec:background}
	
	Throughout, $k$ is an algebraically closed field, $A$ is a finite-dimensional $k$-algebra, and all $A$-modules are finitely generated and assumed to be left modules, unless otherwise stated. Given a second $k$-algebra $B$, we assume that $k$ acts centrally on $A$-$B$-bimodules, and we freely identify $A$-$B$-bimodules with $A \otimes_k \opcat{B}$-modules.
	
	If $M$ is an $A$-module, then we denote by $M^\ast = \Hom_k(M,k)$ the $k$-linear dual of $M$, and by $M^\vee = \Hom_A(M,A)$ the $A$-linear dual of $M$. Both $M^\ast$ and $M^\vee$ are right $A$-modules. If $B$ is a second $k$-algebra and $M$ is an $A$-$B$-bimodule, then $M^\ast$ and $M^\vee$ are $B$-$A$-bimodules.	
	
	We denote by $A\modcat$ the ($k$-linear, abelian) category of all (finitely generated, left) $A$-module and by $\Ch^b(A)$ the (abelian) category of chain complexes of $A$-modules. We use a homological grading convention for chain complexes. We further denote by $D^b(A)$ the (triangulated) bounded derived category of $A\modcat$. Given a second $k$-algebra $B$, we will denote by $\bbimder{A}{B}$ the bounded derived category of complexes of $A$-$B$-bimodules. 
	
	Given a collection of objects $\mathcal{X}$ in $A\modcat$, we denote by $\mathcal{X}\cadd$ the full additive subcategory of $A\modcat$ whose objects are isomorphic to summands of direct sums of objects in $\mathcal{X}$. For example, we have $A\cadd = A\proj$, the (additive) category of all (finitely generated, left) projective $A$-modules. Finally, if $\mathcal{C}$ is an additive category, we denote by $K^b(\mathcal{C})$ the (triangulated) bounded homotopy category of $\mathcal{C}$.

	\subsection{Symmetric algebras}\label{subsec:symmalg}

	Let $A$ be a $k$-algebra. We say that $A$ is \emph{symmetric} if there is an isomorphism of $A$-$A$-bimodules $A \cong A^\ast$. 
	
	We have the following equivalent definitions, \cite[Theorem 3.1]{rickard_symmderived}. 
	
	\begin{thm}\label{thm:symmetric}
		Let $A$ be a $k$-algebra. The following statements are equivalent.
		\begin{enumerate}[label=(\roman*)]
			\item The algebra $A$ is symmetric.
			\item There is a natural isomorphism of contravariant functors $\Hom_k(-,k) \cong \Hom_A(-,A)$	from $A\modcat$ to \emph{$\opcat{A}\modcat$}.
			\item Given $A$-modules $M$ and $P$ such that $P$ is projective, there is an isomorphism of $k$-vector spaces $\Hom_A(M,P) \cong \Hom_A(P,M)^\ast$, functorial in $M$ and $P$.
		\end{enumerate}
	\end{thm}
	
	Further discussion of symmetric algebras can be found in \cite[Section 3]{rickard_symmderived}. An algebra $A$ is symmetric if and only if the opposite algebra $\opcat{A}$ is symmetric. Moreover, $A$ is symmetric if and only if every block of $A$ is symmetric. The property of an algebra being symmetric is preserved by Morita and derived equivalences.
	
	Let $A$ be a finite-dimensional, symmetric $k$-algebra and $P$ a projective $A$-module. Set $E=\opcat{\End_A(P)}$. It is a straightforward exercise to show that $E$ is also a finite-dimensional, symmetric $k$-algebra.

	\subsection{Derived equivalences}
	
	Let $A$ and $B$ be $k$-algebras, and $D^b(A)$, $D^b(B)$ their bounded derived categories. For a refresher on the basics of derived categories and equivalence, the reader is encouraged to see \cite[Chapter 10]{weibel_homalg}.
	
	We recall Rickard's Morita theory for derived categories, \cite[Theorem 6.4]{rickard_derivedmorita}. A \emph{one-sided tilting complex} for $A$ is an object $X$ of $K^b(A\proj)$ which satisfies:
	\begin{itemize}
		\item $\Hom_{D^b(A)}(X,X[i])=0$ whenever $i\neq 0$;
		\item $X\cadd$ generates $K^b(A\proj)$ as a triangulated category.
	\end{itemize}
	
	\begin{thm}\label{thm:derivedmorita}
		There is a derived equivalence \begin{tikzcd}[cramped,sep=small]F:D^b(A) \arrow[r,"\sim"] & D^b(B)\end{tikzcd} if and only if there is a one-sided tilting complex $X$ for $A$ such that $\opcat{\End_{D^b(A)}(X)} \cong B$. 	If such a tilting complex $X$ and equivalence $F$ exist, then $F(X) \cong B$ as $B$-modules.
	\end{thm}
	
	A \emph{standard derived equivalence} between $A$ and $B$ is one of the form
	\begin{center}
		\begin{tikzcd}[sep=small]
			X \ltens_A - : D^b(A) \arrow[r,"\sim"] & D^b(B),
		\end{tikzcd}
	\end{center}
	where $X$ is an object of $\bbimder{B}{A}$. The object $X$ of $\bbimder{B}{A}$ is a \emph{two-sided tilting complex} if $X$ induces a standard derived equivalence as above. 
	
	The following theorem, \cite[Corollary 3.5]{rickard_derivedfunctors}, tells us that we can always replace a derived equivalence by a standard derived equivalence, and it will behave the same on objects.
	
	\begin{thm}\label{thm:standardderived}
		If \begin{tikzcd}[cramped, sep=small]
			F : D^b(A) \arrow[r,"\sim"] & D^b(B)
		\end{tikzcd} is a derived equivalence, then there is a standard derived equivalence
		\begin{tikzcd}[cramped, sep=small]
			X \ltens_A -: D^b(A) \arrow[r,"\sim"] & D^b(B)
		\end{tikzcd} that agrees with $F$ on $A\proj$, and is such that $F(Y) \cong X \ltens_A Y$ for every object $Y$ of $D^b(A)$.
	\end{thm}
	
	We have that $X$ is a two-sided tilting complex if and only if there is an object $\tilde{X}$ of $\bbimder{A}{B}$ such that $X \ltens_A \tilde{X} \cong B$ in $\bbimder{B}{B}$ and $\tilde{X} \ltens_B X \cong A \text{ in } \bbimder{A}{A}$. The object $\tilde{X}$ is called the \emph{inverse} of $X$, \cite[Definition 4.2]{rickard_derivedfunctors}. It is itself a two-sided tilting complex, inducing a standard derived equivalence
	\begin{center}
		\begin{tikzcd}[sep=small]
			\tilde{X} \ltens_B -: D^b(B) \arrow[r,"\sim"] & D^b(A).
		\end{tikzcd}
	\end{center}
	
	Recall that a \emph{perfect} object in a derived category is any object isomorphic to a bounded chain complex of finitely generated projective $A$-modules. If $X \in \bbimder{B}{A}$ is a two-sided tilting complex, then $X$ is perfect in $D^b(B)$ and in $D^b(\opcat{A})$.
	
	The perfect objects in $D^b(A)$ form a triangulated subcategory $\Perf(A)$ of $D^b(A)$. If $A\stmod$ is the stable module category of $A$, then \cite[Theorem 2.1]{rickard_derivedstable} there is an equivalence of triangulated categories \begin{tikzcd}[cramped, sep=small]
		D^b(A) / \Perf(A) \arrow[r,"\sim"] & A\stmod.
	\end{tikzcd}

	\subsection{Perverse equivalences}
	
	Let $A$ and $B$ be algebras, and let $\mathcal{A}=A\modcat$ and $\mathcal{B} = B\modcat$. Suppose that there is a derived equivalence \begin{tikzcd}[cramped,sep=small]
		F:D^b(A) \arrow[r,"\sim"] & D^b(B).
	\end{tikzcd}
	Let $\{S_1,\ldots,S_r\}$, $\{S'_1,\ldots,S'_r\}$ be the sets of isomorphism classes of simple $A$- and $B$-modules respectively, and $I=\{1,\ldots,r\}$ the shared indexing set. It is a standard fact that derived equivalent algebras have the same number of simple modules, but that a derived equivalence need not provide a bijection between the two sets.	
	
	Recall that a full abelian subcategory $\mathcal{A}'$ of $\mathcal{A}$ is a \emph{Serre subcategory} if whenever 
	\begin{center}
		\begin{tikzcd}[sep=small]
			0 \arrow[r] & L \arrow[r] & M \arrow[r] & N \arrow[r] & 0
		\end{tikzcd}
	\end{center}
	is an exact sequence in $\mathcal{A}$, the object $M$ belongs to $\mathcal{A}'$ if and only if both $L$ and $N$ belong to $\mathcal{A}'$. Serre subcategories of $\mathcal{A} = A\modcat$ coincide with subsets $J \subset I$. Given a subset $J \subset I$, we denote by $\mathcal{A}_J$ the Serre subcategory \emph{generated by} the set $\{S_j\}_{j \in J}$: this is the full subcategory of $\mathcal{A}$ whose objects are precisely the $A$-modules whose composition factors are all in the set $\{S_j\}_{j \in J}$.
	
	If $\mathcal{A}'$ is a Serre subcategory of $\mathcal{A}$, then we define $D^b_{\mathcal{A}'}(A)$ to be the thick subcategory of $D^b(A)$ whose objects are those isomorphic to complexes with homology contained entirely in $\mathcal{A}'$. The (Verdier) quotient $D^b(A) / D^b_{\mathcal{A}'}(A)$ is the triangulated category obtained from $D^b(A)$ by localising at the collection of morphisms $f:U \to V$ in $D^b(A)$ for which $\cone(f) \in D^b_{\mathcal{A}'}(A)$, together with a triangulated functor \begin{tikzcd}[cramped,sep=small]
		D^b(A) \arrow[r] & D^b(A) / D^b_{\mathcal{A}'}(A),
	\end{tikzcd} universal among all triangulated functors \begin{tikzcd}[cramped,sep=small]
		D^b(A) \arrow[r] & \mathcal{T}
	\end{tikzcd} whose kernel contains $D^b_{\mathcal{A}'}(A)$.
	
	Suppose we have filtrations
	\begin{align*}
		\emptyset &= I_0 \subset I_1 \subset \ldots \subset I_t = I, \\
		\emptyset &= I'_0 \subset I'_1 \subset \ldots \subset I'_t = I.
	\end{align*}
	This defines filtrations by Serre subcategories
	\begin{align*}
		0 &= \mathcal{A}_0 \subset \mathcal{A}_1 \subset \ldots \subset \mathcal{A}_t = \mathcal{A}, \\
		0 &= \mathcal{B}_0 \subset \mathcal{B}_1 \subset \ldots \subset \mathcal{B}_t = \mathcal{B},
	\end{align*}
	with $\mathcal{A}_i \defeq \mathcal{A}_{I_i}$ and $\mathcal{B}_i \defeq \mathcal{B}_{I'_i}$. For each $i$, the \emph{Serre quotient} $\mathcal{A}_i / \mathcal{A}_{i-1}$ is an abelian category, together with an exact functor \begin{tikzcd}[sep=small, cramped] \mathcal{A}_i \arrow[r,"\sim"] & \mathcal{A}_i / \mathcal{A}_{i-1}\end{tikzcd}, universal among all exact functors \begin{tikzcd}[sep=small, cramped] \mathcal{A}_i \arrow[r,"\sim"] & \mathcal{C}\end{tikzcd} whose kernel contains $\mathcal{A}_{i-1}$. Let $p:\{1,\ldots,t\} \to \mathbb{Z}$ be a function. 
	
	\begin{defn}\label{defn:perverse_gen}
		The derived equivalence \begin{tikzcd}[cramped,sep=small]F:D^b(A) \arrow[r,"\sim"] & D^b(B)\end{tikzcd} is \emph{perverse relative to} $(\mathcal{A}_\bullet, \mathcal{B}_\bullet, p)$ if both of the following hold.
		\begin{enumerate}[label=(\roman*)]
			\item The functor $F$ restricts to an equivalence \begin{tikzcd}[sep=small, cramped] D^b_{\mathcal{A}_i}(A) \arrow[r,"\sim"] & D^b_{\mathcal{B}_i}(B)\end{tikzcd} of triangulated categories for every $i$.
			\item For every $i$, $F[p(i)]$ induces an equivalence \begin{tikzcd}[sep=small, cramped] \mathcal{A}_i/\mathcal{A}_{i-1} \arrow[r,"\sim"] & \mathcal{B}_i/\mathcal{B}_{i-1}\end{tikzcd} of abelian categories.
		\end{enumerate}
		We call $(\mathcal{A}_\bullet, \mathcal{B}_\bullet, p)$ the \emph{perversity data} and $p$ the \emph{perversity function} of the perverse equivalence $F$.
	\end{defn}
	
	We may also say that $F$ is perverse relative to the data $(I_\bullet, I'_\bullet, p)$. Notationally, to highlight the shifts, we may write the filtrations as
	\[
	\emptyset = I_0 \subset_{p(1)} I_1 \subset_{p(2)} \ldots \subset_{p(t)} I_t = I,
	\]
	\[
	\emptyset = I'_0 \subset_{p(1)} I'_1 \subset_{p(2)} \ldots \subset_{p(t)} I'_t = I.
	\]
	
	Digging a little deeper into the definition, the equivalence \begin{tikzcd}[sep=small, cramped] D^b_{\mathcal{A}_i}(A) \arrow[r,"\sim"] & D^b_{\mathcal{B}_i}(B)\end{tikzcd} for each $i$ induces an equivalence \begin{tikzcd}[sep=small, cramped] D^b(A) / D^b_{\mathcal{A}_i}(A) \arrow[r,"\sim"] & D^b(B) / D^b_{\mathcal{B}_i}(B)\end{tikzcd}. The required equivalence in condition (ii) can be seen in the commutative diagram (see e.g. \cite[Remark 3.21]{grant_periodic})
	\begin{center}
		\begin{tikzcd}
			D^b(A)/D^b_{\mathcal{A}_{i-1}}(A) \arrow[r,"\sim"] & D^b(B)/D^b_{\mathcal{B}_{i-1}}(B) \\
			\mathcal{A}_i/\mathcal{A}_{i-1} \arrow[r,dashed,"\exists\sim"] \arrow[u,hook] & \mathcal{B}_i/\mathcal{B}_{i-1} \arrow[u,hook]
		\end{tikzcd}
	\end{center}
	sitting inside a larger diagram
	\begin{center}
		\begin{tikzcd}[sep=small]
			& D^b(A) \arrow[rr,"F{[p(i)]}"] \arrow[dd] & & D^b(B) \arrow[dd] \\
			\mathcal{A}_i \arrow[ur,hook] \arrow[dd] & & \mathcal{B}_i \arrow[ur,hook]  \\
			& D^b(A)/D^b_{\mathcal{A}_{i-1}}(A) \arrow[rr,"\sim"{near start}] & & D^b(B)/D^b_{\mathcal{B}_{i-1}}(B) \\
			\mathcal{A}_i/\mathcal{A}_{i-1} \arrow[rr,dashed,"\exists\sim"] \arrow[ur,hook] & & \mathcal{B}_i/\mathcal{B}_{i-1} \arrow[from=uu,crossing over] \arrow[ur,hook]
		\end{tikzcd}
	\end{center}
	whose vertical arrows are quotient functors, and the embedding \begin{tikzcd}[cramped,sep=small] \mathcal{A}_i \arrow[r,hook] & D^b(A)\end{tikzcd} is via the usual embedding of $A\modcat$ in $D^b(A)$. The functor \begin{tikzcd}[cramped,sep=small] \mathcal{A}_i/\mathcal{A}_{i-1} \arrow[r,hook] & D^b(A)/D^b_{\mathcal{A}_{i-1}}(A) \end{tikzcd} exists and is fully faithful by the universal property of the quotient \begin{tikzcd}[cramped,sep=small] \mathcal{A}_i \arrow[r] & \mathcal{A}_i/\mathcal{A}_{i-1} \end{tikzcd}. Identical reasoning justifies the arrows on the other side. 
	
	Two immediate consequences of the definition are the following, \cite[Lemma 4.16, Lemma 4.2]{chuang-rouquier_perverse17}.
	
	\begin{lem}\label{lem:perversity0}
		If \begin{tikzcd}[cramped,sep=small]F:D^b(A) \arrow[r,"\sim"] & D^b(B)\end{tikzcd} is perverse with perversity function $p \equiv 0$, then we have an equivalence of abelian categories \begin{tikzcd}[cramped,sep=small]A\modcat \arrow[r,"\sim"] & B\modcat.\end{tikzcd}
	\end{lem}
	
	\begin{lem}\label{lem:perverseinv}
		If \begin{tikzcd}[cramped,sep=small]F:D^b(A) \arrow[r,"\sim"] & D^b(A')\end{tikzcd} is perverse relative to $(\mathcal{A}_\bullet, \mathcal{B}_\bullet, p)$, then the inverse \begin{tikzcd}[cramped,sep=small]F^{-1}:D^b(B) \arrow[r,"\sim"] & D^b(A)\end{tikzcd} is perverse relative to $(\mathcal{B}_\bullet, \mathcal{A}_\bullet, -p)$.
	\end{lem}
	
	In general, the composition of two perverse equivalences need not remain perverse. However, the composition of two perverse equivalences at a fixed middle filtration is perverse, \cite[Lemma 4.4]{chuang-rouquier_perverse17}.
	
	\begin{lem}\label{lem:perversecomp}
		Let $C$ be another algebra, derived equivalent to $A$ and $B$. Suppose we have an equivalence 
		\begin{tikzcd}[cramped,sep=small]
			F: D^b(A) \arrow[r,"\sim"] & D^b(B),
		\end{tikzcd}
		perverse relative to $(\mathcal{A}_\bullet, \mathcal{B}_\bullet, p)$, and an equivalence
		\begin{tikzcd}[cramped,sep=small]
			F': D^b(B) \arrow[r,"\sim"] & D^b(C),
		\end{tikzcd}
		perverse relative to $(\mathcal{B}_\bullet, \mathcal{C}_\bullet, p')$. Then the composition $F' \circ F$ is perverse relative to $(\mathcal{A}_\bullet, \mathcal{C}_\bullet, p+p')$.
	\end{lem}
	
	The preceding results produce the following, \cite[Proposition 4.17]{chuang-rouquier_perverse17}.
	
	\begin{prop}\label{prop:auxilperv}
		Suppose we have equivalences
		\begin{tikzcd}[cramped,sep=small]
			F: D^b(A) \arrow[r,"\sim"] & D^b(B),
		\end{tikzcd}
		perverse relative to $(\mathcal{A}_\bullet, \mathcal{B}_\bullet, p)$, and
		\begin{tikzcd}[cramped,sep=small]
			F': D^b(A) \arrow[r,"\sim"] & D^b(C),
		\end{tikzcd}
		perverse relative to $(\mathcal{A}_\bullet, \mathcal{C}_\bullet, p)$. Then $F' \circ F^{-1}$ induces an equivalence of abelian categories \begin{tikzcd}[cramped,sep=small]
			B\modcat \arrow[r,"\sim"] & C\modcat.
		\end{tikzcd}
	\end{prop}
	
	\begin{proof}
		By Lemmas \ref{lem:perverseinv} and \ref{lem:perversecomp}, $F' \circ F^{-1}$ is perverse relative to $(\mathcal{B}_\bullet, \mathcal{C}_\bullet, \overbar{p})$, with $\overbar{p} = p + (-p) \equiv 0$, so the result follows by Lemma \ref{lem:perversecomp}.
	\end{proof}
	
	In other words, a perverse equivalence is determined up to Morita equivalence by the filtration on the left hand side and the perversity function $p$. We think of this as a \emph{uniqueness} result for perverse equivalences.
	
	As one final basic definition, an equivalence \begin{tikzcd}[cramped,sep=small]
		F: D^b(A) \arrow[r,"\sim"] & D^b(A),
	\end{tikzcd} perverse relative to $(\mathcal{A}_\bullet, \mathcal{A}_\bullet, p)$, is a \emph{self-perverse equivalence}. In such cases, we will say that $F$ is perverse relative to $(\mathcal{A}_\bullet, p)$.

	\subsection{Simple modules and projective modules}
	
	For the remainder of this section, we assume our algebras are symmetric. Here, it will be beneficial to consider the perversity data $(I_\bullet, I'_\bullet, p)$. Fix the notation of the previous subsection. Let $\mathcal{S}$ and $\mathcal{S}'$ be the sets of isomorphism classes of simple $A$- and $B$-modules respectively. For a subset $J \subset I$, set $\mathcal{S}_J = \{S_j\}_{j \in J}$ and $\mathcal{S}'_J = \{S'_j\}_{j \in J}$. Given filtrations $I_\bullet$, $I'_\bullet$ of subsets of $I$, set $\mathcal{S}_i = \mathcal{S}_{I_i}$ and $\mathcal{S}'_i = \mathcal{S}'_{I_i}$.
	
	We can rephrase the conditions in Definition \ref{defn:perverse_gen} to conditions wholly on the simple modules themselves. The following is \cite[Lemma 4.19]{chuang-rouquier_perverse17}.
	
	\begin{prop}\label{prop:perversesimp}
		A derived equivalence \begin{tikzcd}[cramped, sep=small]F:D^b(A) \arrow[r,"\sim"] & D^b(B)\end{tikzcd} is perverse relative to $(I_\bullet, I'_\bullet, p)$ if both of the following hold.
		\begin{enumerate}[label=(\roman*)]
			\item For every $i$ and every $V \in \mathcal{S}_i\setminus\mathcal{S}_{i-1}$, the composition factors of $H_t(F(V))$ for $t\ne p(i)$ are all in $\mathcal{S}'_{i-1}$, and there is a filtration $L_1 \subset L_2 \subset H_{p(i)}(F(V))$ such that the composition factors of $L_1$ and the composition factors of $H_{p(i)}(F(V))/L_2$ are all in $\mathcal{S}'_{i-1}$, and $L_2/L_1 \in \mathcal{S}'_i\setminus\mathcal{S}'_{i-1}$.
			\item The map $V\mapsto L_2/L_1$ described above is a bijection between $\mathcal{S}_i\setminus\mathcal{S}_{i-1}$ and $\mathcal{S}'_i\setminus\mathcal{S}'_{i-1}$.
		\end{enumerate}
	\end{prop}
	
	In fewer words, for every $V \in \mathcal{S}_i\setminus\mathcal{S}_{i-1}$, the composition factors of $H_t(\Phi(V))$ are all in $\mathcal{S}'_{i-1}$, except for a single composition factor of $H_{p(i)}(\Phi(V))$, which lies in $\mathcal{S}'_i\setminus\mathcal{S}'_{i-1}$. 
	
	Perhaps the most significant consequence of Proposition \ref{prop:perversesimp} is that, when $F$ is a perverse equivalence, unlike for a general perverse equivalence, $F$ induces a bijection $\mathcal{S}_i \setminus \mathcal{S}_{i-1} \leftrightarrow \mathcal{S}'_i \setminus \mathcal{S}'_{i-1}$ between the layers of the two filtrations of subsets of simple modules. Gluing these stratified bijections together therefore defines a bijection between $\mathcal{S}$ and $\mathcal{S}'$. This is one sense in which a perverse equivalence gives us more information than an arbitrary equivalence.
	
	One can also rephrase the conditions for perversity in terms of projective modules. The following is taken from \cite[Lemmas 4.7, 4.21]{chuang-rouquier_perverse17}.
	
	Suppose \begin{tikzcd}[cramped, sep=small]F:D^b(A) \arrow[r,"\sim"] & D^b(B)\end{tikzcd} is a perverse equivalence, relative to $(I_\bullet, I'_\bullet, p)$. For each $i$, let $\mathcal{P}_i$ be the set of projective indecomposable $A$-modules $P_V$ corresponding to the simple modules $V \in \mathcal{S} \setminus \mathcal{S}_{t-i}$, and $\mathcal{P}'_i$ the set of projective indecomposable $B$-modules $P_{V'}$ corresponding to the simple modules $V' \in \mathcal{S}' \setminus \mathcal{S}'_{t-i}$. This defines filtrations
	\[
	\emptyset = \mathcal{P}_0 \subset \mathcal{P}_1 \subset \ldots \subset \mathcal{P}_t = \mathcal{P},
	\]
	\[
	\emptyset = \mathcal{P}'_0 \subset \mathcal{P}'_1 \subset \ldots \subset \mathcal{P}'_t = \mathcal{P}'
	\]
	of the sets $\mathcal{P}$ and $\mathcal{P}'$ of projective $A$- and $B$-modules respectively. Define a function $\overbar{p}:\{1,\ldots,r\} \to \mathbb{Z}$ by $\overbar{p}(i) = p(t-i+1)$. 
	
	\begin{prop}\label{prop:perverseproj}
		A derived equivalence \begin{tikzcd}[cramped, sep=small]F:D^b(A) \arrow[r,"\sim"] & D^b(B)\end{tikzcd} is perverse relative to $(I_\bullet, I'_\bullet, \overbar{p})$ if and only if the following hold.
		\begin{enumerate}[label=(\roman*)]
			\item For every $i$ and every indecomposable projective $A$-module $P \in \mathcal{P}_i \setminus \mathcal{P}_{i-1}$, the object $\Phi(P)$ is isomorphic in $D^b(B)$ to a complex $X$ of projective $B$-modules, such that every term of $X$ is a direct sum of modules in $\mathcal{P}'_{i-1}$, except in degree $\overbar{p}(i)$, which has exactly one indecomposable summand in $\mathcal{P}'_i \setminus \mathcal{P}'_{i-1}$, say $P'$, with all others in $\mathcal{P}'_{i-1}$.
			\item The map $P \mapsto P'$ as above defines a bijection between $\mathcal{P}_i \setminus \mathcal{P}_{i-1}$ and $\mathcal{P}'_i \setminus \mathcal{P}'_{i-1}$.
		\end{enumerate}
	\end{prop}
	
	In particular, gluing these stratified bijections produces a bijection $\mathcal{P} \leftrightarrow \mathcal{P}'$, matching (in reverse) the bijection $\mathcal{S} \leftrightarrow \mathcal{S}'$.

	\subsection{Two-step perverse equivalences}

	Let $A$, $B$ be as before and suppose we have a derived equivalence \begin{tikzcd}[cramped,sep=small]
		F: D^b(A) \arrow[r,"\sim"] & D^b(B).
	\end{tikzcd} Suppose there are filtrations
	\[
	\emptyset \subset_{p(1)} J \subset_{p(2)} I, 
	\]
	\[
	\emptyset \subset_{p(1)} J' \subset_{p(2)} I
	\]
	such that $F$ is perverse relative to $(I_\bullet,I'_\bullet, p)$. Then we say that $F$ is a \emph{two-step perverse equivalence of width $d$}, where $d = p(2) - p(1)$.
	
	Let $M$ be an $A$-module. Take a projective cover \begin{tikzcd}[cramped, sep=small]P(M) \arrow[r,"\pi_M"] & M\end{tikzcd} of $M$. Denote by $M_J$ the largest quotient of $P(M)$ by a submodule of $\ker(\pi_M)$ such that all composition factors of the kernel of the induced map \begin{tikzcd}[cramped,sep=small]M_J \arrow[r] & M\end{tikzcd} are in $J$. Let \begin{tikzcd}[cramped,sep=small]Q_{M,J} \arrow[r,"\pi_{M,J}"] & \ker(\phi_{M,J})\end{tikzcd} be a projective cover of the kernel of the canonical map \begin{tikzcd}[cramped,sep=small]P(M) \arrow[r,"\phi_{M,J}"] & M_J\end{tikzcd}.
	
	\begin{defn}\label{defn:combtilt}
		Given $J \subset I$, the \emph{combinatorial tilting complex at $J$} is the complex
		\[
		T=\bigoplus_{j \in J} T_j \oplus \bigoplus_{i \in I \setminus J}P_i[1],
		\]
		where, for $j \in J$, $T_j$ is the complex
		\begin{center}
			\begin{tikzcd}[sep=small]
				0 \arrow[r] & Q_{S_j,J} \arrow[r] & P_j \arrow[r] & 0,
			\end{tikzcd}
		\end{center}
		concentrated in degrees 1 and 0.
	\end{defn}
	
	We note that Grant \cite[Definition 5.4]{grant_periodic} allows $T=\bigoplus_{i \in I}T_i^{\ell_i}$, where $\ell_i \ge 1$ for all $i$, with $T_i=P_i[1]$ for $i \in I \setminus J$. The complex $T$ with $\ell_i = 1$ is the \emph{basic} combinatorial tilting complex at $J$. 
	
	When $J$ is such that $\Ext^1_A(S_i,S_j) = 0$ for every $i,j \in J$, then for $j \in J$, the complex $T_j$ is \begin{tikzcd}[cramped, sep=small]
		P(\rad(P_j)) \arrow[r,"\pi"] & P_j,
	\end{tikzcd} where \begin{tikzcd}[cramped,sep=small]P(\rad(P_j)) \arrow[r,"\pi"] & \rad(P_j)\end{tikzcd} is a projective cover of $\rad(P_i)$. That is, $P(\rad(P_j)) = \bigoplus_{i \in I \setminus J} P_i \otimes_k \Ext^1_A(S_i,S_j)$, where $\Ext^1_A(S_i,S_j)$ is the \emph{multiplicity module}.
	
	Combinatorial tilting complexes were introduced by Rickard \cite{rickard_thesis} for $J$ a single index, generalised to arbitrary subsets $J$ by Okuyama \cite{okuyama_method}. They are also called \emph{Okuyama-Rickard two-term tilting complexes}, among other names. Combinatorial tilting complexes have wide-reaching applications, for example in silting theory \cite{aihara-iyama_silting} and cluster tilting theory \cite{BMRRT_clustertilting}.
	
	Given a subset $J$, the basic combinatorial tilting complex $T$ at $J$ exists and is unique up to isomorphism, \cite[Lemma 5.5, Corollary 5.7]{grant_periodic}. Further. $T$ is a tilting complex, \cite[Proposition 1.1]{okuyama_method}, \cite[Proposition 5.6]{grant_periodic}. Thus, by Theorem \ref{thm:derivedmorita}, given a combinatorial tilting complex $T$, there is an algebra $B = \opcat{\End_{D^b(A)}(T)}$ such that there is a derived equivalence \begin{tikzcd}[cramped,sep=small]F_J:D^b(A) \arrow[r,"\sim"] & D^b(B)\end{tikzcd}. We call $F_J$ the \emph{combinatorial tilt of $A$ at $J$}.
	
	Two-step perverse equivalences of width $-1$ coincide with combinatorial tilts. The following can be found in \cite[Proposition 5.3]{chuang-rouquier_perverse17}.
	
	\begin{prop}\label{prop:combtiltperverse}
		Let $A$ be a finite-dimensional symmetric $k$-algebra. Let $I$ be an indexing set of the isomorphism classes of simple $A$-modules. Given $J \subset I$, the combinatorial tilt \begin{tikzcd}[cramped,sep=small]F_J:D^b(A) \arrow[r,"\sim"] & D^b(B)\end{tikzcd} at $J$ is a perverse equivalence, with filtrations both given by $\emptyset \subset_0 J \subset_{-1} I$.
	\end{prop}
	
	Chuang and Rouquier call these combinatorial tilts \emph{elementary perverse equivalences}.

	\subsection{Standard equivalences}\label{subsec:standardperverse}

	Suppose \begin{tikzcd}[cramped,sep=small]F:D^b(A) \arrow[r,"\sim"] & D^b(B)\end{tikzcd} is a derived equivalence. Recall by Theorem \ref{thm:standardderived} that there is a complex $X$ of $B$-$A$-bimodules such that, for every $V \in D^b(A)$, $F(V) \cong X \ltens_A V$ in $D^b(B)$. 
	
	\begin{prop}\label{prop:perversestandard}
		If the equivalence $F$ is perverse relative to $(I_\bullet, I'_\bullet, p)$, then the equivalence \begin{tikzcd}[cramped,sep=small]
			X \ltens_A -: D^b(A) \arrow[r,"\sim"] & D^b(B)
		\end{tikzcd} is perverse, with the same perversity data.
	\end{prop}
	
	\begin{proof}
		By Proposition \ref{prop:perversesimp}, the perversity of the derived equivalence $X \ltens_A -$ depends only on the images of simple $A$-modules. But for every simple $A$-module $S_i$, $X \ltens_A S_i \cong F(S_i)$. The result follows.
	\end{proof}
	
	Let $X$ be as above. Set $X^\vee=\rhom_B(X,B)$, a complex of $A$-$B$-bimodules. Then by \cite[Proposition 4.1]{rickard_derivedfunctors}, \begin{tikzcd}[cramped, sep=small]
		X^\vee \ltens_B -: D^b(B) \arrow[r,"\sim"] & D^b(A)
	\end{tikzcd} is a derived equivalence, mutually inverse with the equivalence	$X \ltens_A -$, and is the standard derived equivalence agreeing with $F^{-1}$ on objects of $D^b(B)$. By Lemma \ref{lem:perverseinv} and Proposition \ref{prop:perverseproj}, $X^\vee \ltens_B -$ and $F^{-1}$ are both perverse relative to $(I'_\bullet, I_\bullet, -p)$. 
	
	The two-sided tilting complexes $X$ and $X^\vee$ also induce perverse equivalences on the derived categories of right modules. The following combines \cite[Lemma 4.3]{rickard_derivedfunctors} and \cite[Lemma 4.20]{chuang-rouquier_perverse17}.
	
	\begin{prop}\label{prop:rightmodperverse}
		The functor \begin{tikzcd}[cramped, sep=small]
			- \ltens_B X: D^b(\opcat{B}) \arrow[r] & D^b(\opcat{A})
		\end{tikzcd} is an equivalence, and is perverse relative to $(I_\bullet, I'_\bullet, -p)$. Similarly, \begin{tikzcd}[cramped,sep=small]
			- \ltens_A X^\vee: D^b(\opcat{A}) \arrow[r] & D^b(\opcat{B})
		\end{tikzcd} is an equivalence, perverse relative to $(I'_\bullet, I_\bullet, p)$. Moreover, these two equivalences are mutually inverse.
	\end{prop}
	
	Thus, the equivalences $F$ and $F^{-1}$ induce equivalences \begin{tikzcd}[cramped, sep=small]
		\tilde{F}: D^b(\opcat{B}) \arrow[r] & D^b(\opcat{A}),
	\end{tikzcd} perverse relative to $(I_\bullet, I'_\bullet, -p)$, and \begin{tikzcd}[cramped, sep=small]
		\tilde{F}^{-1}: D^b(\opcat{A}) \arrow[r] & D^b(\opcat{B}),
	\end{tikzcd}  perverse relative to $(I'_\bullet, I_\bullet, p)$.

	\section{Periodic Modules and Algebras}\label{sec:periodicmodulesalgebras}

	\subsection{Periodic modules}

	Let $E$ be a finite-dimensional $k$-algebra and $M$ an $E$-module. If $\sigma$ is an automorphism of $E$, then we define the \emph{twisted module} ${}_{\sigma\!}M$ to be the $E$-module with $E$-action $x \cdot m = \sigma(x)m$ for $x \in E$, $m \in M$. We can adapt this definition for right modules or bimodules in the obvious way. For any $E$-module $M$, we have an isomorphism ${}_{\sigma\!}M \cong {}_{\sigma\!}E \otimes_E M$ of $E$-modules.
	
	Recall that the Heller translate of $M$ is $\Omega_E(M) = \ker(\pi_M)$, where \begin{tikzcd}[cramped,sep=small]P(M) \arrow[r,"\pi_M"] & M\end{tikzcd} is a projective cover of $M$. One can iterate this construction: for $n \ge 1$, we set
	\[
	\Omega_E^{n+1}(M) = \Omega_E(\Omega_E^n(M)).
	\]
	
	\begin{defn}
		The $E$-module $M$ is \emph{$\sigma$-periodic of period $n$} if there is an automorphism $\sigma$ of $E$ and an $n\ge 1$ such that $\Omega_E^n(M) \cong {}_{\sigma\!}M$.
	\end{defn}
	
	That is, $M$ is $\sigma$-periodic of period $n$ if there is an exact sequence
	\begin{center}
		\begin{tikzcd}[sep=small]
			0 \arrow[r] & {}_{\sigma\!}M \arrow[r] & P_{n-1} \arrow[r] & \ldots \arrow[r] & P_1 \arrow[r] & P_0 \arrow[r] & M \arrow[r] & 0
		\end{tikzcd}
	\end{center}
	of $E$-modules such that each $P_i$ is projective. We call the complex
	\begin{center}
		\begin{tikzcd}[sep=small]
			P_{n-1} \arrow[r] & \ldots \arrow[r] & P_1 \arrow[r] & P_0
		\end{tikzcd}
	\end{center}
	a \emph{truncated projective resolution of $M$}. If $\sigma =\id$, then we say simply that $M$ is \emph{periodic}. 
	
	We note that, if there is some $n$ such that $M$ is $\sigma$-periodic of period $n$ for some automorphism $\sigma$, then there must exist some minimal such $n$. We emphasise that we do not demand minimality, as it may sometimes be expedient to consider different periodicities of the same $M$ at once. We also note that there may be different automorphisms $\sigma$ of $E$ for which $M$ is $\sigma$-periodic.
	
	There is an obvious dual definition for right modules. The $\opcat{E}$-module $N$ is \emph{$\tau$-periodic of period $n$} for an automorphism $\tau$ of $E$ and $n \in \mathbb{Z}_+$ if $\Omega^n_{\opcat{E}}(N) \cong N_\tau$.

	\subsection{Periodic and relatively periodic algebras}

	Let $E$ be a finite-dimensional $k$-algebra.
	
	\begin{defn}
		We say that $E$ is \emph{$\sigma$-periodic of period $n$} if there is an automorphism $\sigma$ of $E$ and an $n \ge 1$ such that $E$ is $\sigma \otimes \id_E$-periodic of period $n$ as an $E \otimes \opcat{E}$-module.
	\end{defn}
	
	That is, $E$ is $\sigma$-periodic of period $n$ if there is an exact sequence of $E$-$E$-bimodules
	\begin{center}
		\begin{tikzcd}[sep=small]
			0 \arrow[r] & _{\sigma\!}E \arrow[r] & Y_{n-1} \arrow[r] & \ldots \arrow[r] & Y_1 \arrow[r] & Y_0 \arrow[r] & E \arrow[r] & 0 
		\end{tikzcd}
	\end{center}
	such that each $Y_i$ is projective as an $E$-$E$-bimodule.
	
	A survey of symmetric algebras with this property can be found in \cite{erdmann-skowronski_periodic}. It remains an interesting open problem to classify the finite-dimensional periodic algebras.
	
	If $E$ is a finite-dimensional $k$-algebra and there exists an automorphism $\sigma$ of $E$ such that $E$ is $\sigma$-periodic of period $n$, then every $E$-module $M$ is $\sigma$-periodic of period $n$. Indeed, one can apply the functor $- \otimes_E M$ to the exact sequence of projective $E$-$E$-bimodules above, to obtain an exact sequence
	\begin{center}
		\begin{tikzcd}[sep=small]
			0 \arrow[r] & _{\sigma\!}M \arrow[r] & Y_{n-1} \otimes_E M \arrow[r] & \ldots \arrow[r] & Y_1 \otimes_E M \arrow[r] & Y_0 \otimes_E M \arrow[r] & M \arrow[r] & 0,
		\end{tikzcd}
	\end{center}
	where every term $Y_i \otimes_E M$ is projective as an $E$-module. This is thus a truncated projective resolution of $M$, and we have $\Omega_E^n(M) \cong {}_{\sigma\!}M$.
	
	We also have a notion of relative periodicity.
	
	\begin{defn}
		Let $B$ be a subalgebra of $E$. Then $E$ is $\sigma$-periodic of period $n$ \emph{relative to $B$} if there is an exact sequence of $E$-$E$-bimodules
		\begin{center}
			\begin{tikzcd}[sep=small]
				0 \arrow[r] & _{\sigma\!}E \arrow[r] & Y_{n-1} \arrow[r] & \ldots \arrow[r] & Y_1 \arrow[r] & Y_0 \arrow[r] & E \arrow[r] & 0 
			\end{tikzcd}
		\end{center}
		such that each $Y_i$ is a direct summand of the $E$-$E$-bimodule $E \otimes_B E$.
	\end{defn}
	
	Setting $B = k$, we recover the usual notion of $\sigma$-periodicity.
	
	\subsection{A theorem of Grant}
	
	The pre-established link between periodicity and perversity for symmetric algebras is the following result of Grant, \cite[Theorem 3.9, Proposition 3.22]{grant_periodic}, \cite[Theorem 4.3]{grant_thesis}.
	
	\begin{thm}\label{thm:grantperiodic}
		Let $A$ be a finite-dimensional symmetric $k$-algebra. Let $P$ be a projective $A$-module and set $E=\opcat{\End_A(P)}$. If there is subalgebra $B$ of $E$ such that $P$ is projective as a $\opcat{B}$-module and $P^\vee$ is projective as a $B$-module, and there is an automorphism $\sigma$ of $E$ and $n \in \mathbb{Z}_+$ such that $E$ is $\sigma$-periodic of period $n$ relative to $B$, then there is a two-step self-perverse equivalence \begin{tikzcd}[cramped, sep=small]
			\Psi_P : D^b(A) \arrow[r,"\sim"] & D^b(A)
		\end{tikzcd}  of width $n$.
	\end{thm}
	
	We call the equivalence $\Psi_P$ the \emph{(relative) periodic twist at $P$}.
	
	Grant's key examples in the non-relative case are symmetric algebra analogues of the spherical twists of Seidel and Thomas \cite{seidel-thomas_spherical} and the $\mathbb{P}^n$-twists of Huybrechts and Thomas \cite{huybrechts-thomas_Pn}, for which $E \cong k[x]/\langle x^{n+1} \rangle$. In the relative case, Grant gives the example of toric twists, for which $E \cong k[x,y] / \langle x^2,y^2 \rangle$.
	
	The next section will be devoted to tightening Theorem \ref{thm:grantperiodic} into an if-and-only-if statement, which we will achieve by working more intimately with the derived category.

	\section{Strongly Periodic Modules and Perverse Equivalences}\label{sec:main}
	
	We come now to our main result, Theorem \ref{thm:main}. This theorem gives necessary and sufficient conditions for the existence of a two-step self-perverse equivalence \begin{tikzcd}[sep=small]
		\Psi: D^b(A) \arrow[r,"\sim"] & D^b(A),
	\end{tikzcd} arising as the cone of a map $A \to X$ in $\bbimder{A}{A}$.
	
	We will fix some notation throughout this section. We will always denote by $A$ a finite-dimensional symmetric $k$-algebra, $\{S_1,\ldots,S_r\}$ a complete set of simple $A$-modules up to isomorphism, $\{P_1,\ldots,P_r\}$ the set of projective indecomposable $A$-modules such that $P_i/\rad(P_i) \cong S_i \cong \soc(P_i)$, and $I=\{1,\ldots,r\}$ an indexing set for these modules.
	
	We will be working with endomorphism algebras. If $P$ is a projective $A$-module and $E=\opcat{\End_A(P)}$, then $P$ has the structure of an $A$-$E$-bimodule. Given $J \subset I$ and integers $m_j \in \mathbb{Z}_+$ such that $P \cong \bigoplus_{j\in J}P_j^{m_j}$, we have that $E \cong \bigoplus_{j \in J}\Hom_A(P,P_j)^{m_j}$ as $E$-modules. In particular, each $\Hom_A(P,P_j)$ is a projective $E$-module, and the functor \begin{tikzcd}[cramped,sep=small]
		\Hom_A(P,-): A\modcat \arrow[r] & E\modcat
	\end{tikzcd} restricts to an equivalence of additive categories \begin{tikzcd}[cramped,sep=small]
		\Hom_A(P,-): P\cadd \arrow[r,"\sim"] & E\proj.
	\end{tikzcd}

	\subsection{Strong periodicity}\label{subsec:mainthm}
	
	Before stating our main theorem, it is necessary that we make the following definition.
	
	\begin{defn}\label{defn:strongperiod}
		Let $E$ be a finite-dimensional self-injective $k$-algebra, $M$ an $E$-module, $\sigma$ an automorphism of $E$, $n \in \mathbb{Z}_+$, and $\alpha \in \Ext^n_{E \otimes_k \opcat{E}}(E,{}_{\sigma\!}E)$. We say that $M$ is \emph{strongly $\sigma$-periodic of period $n$ relative to $\alpha$} if $\alpha \ltens_E M$ induces an isomorphism $\Omega_E^n(M) \cong {}_{\sigma\!}M$. Dually, we say that an $\opcat{E}$-module $N$ is \emph{strongly $\tau$-periodic of period $n$ relative to $\alpha$}, for $\tau$ an automorphism of $E$, if there exists some $\alpha \in \Ext^n_{E \otimes_k \opcat{E}}(E,E_\tau)$ such that $N \ltens_E \alpha$ induces an isomorphism $\Omega_{\opcat{E}}^n(N) \cong N_{\tau}$.
	\end{defn}
	
	It is worth taking the time to unpack this definition. We have that $\alpha \ltens_E M \in \Ext^n_E(M,{}_{\sigma\!}M)$, and 
	\[
	\Ext^n_E(M,{}_{\sigma\!}M) \cong \Hom_{D^b(E)}(M,{}_{\sigma\!}M[n]) \cong \Hom_{E\stmod}(M,\Omega^{-n}_E({}_{\sigma\!}M)).
	\]
	The extension $\alpha$ therefore induces an $E$-module homomorphism $M \to \Omega^{-n}_E({}_{\sigma\!}M)$. Since $E$ is self-injective, this in turn induces an $E$-module homomorphism $\Omega_E^n(M) \to {}_{\sigma\!}M$. The condition in Definition \ref{defn:strongperiod} is that this induced $E$-module homomorphism is an isomorphism. In particular, the $E$-module $M$ is $\sigma$-periodic of period $n$.
	
	The following alternative characterisation of strong $\sigma$-periodicity will be extremely useful.
	
	Recall that there is an equivalence of triangulated categories \begin{tikzcd}[cramped, sep=small]
		A\stmod \arrow[r,"\sim"] & D^b(A) / \Perf(A).
	\end{tikzcd} Given an object $X \in D^b(A)$, we denote by $\overbar{X}$ the image of $X$ under the quotient map \begin{tikzcd}[cramped, sep=small]
		D^b(A) \arrow[r] & D^b(A) / \Perf(A).
	\end{tikzcd} There is an $A$-module $W$ such that $W \cong \overbar{X}$ in $A\stmod$, under the above equivalence. We will freely identify these in what follows.
	
	Let $\alpha \in \Ext_{E \otimes_k \opcat{E}}^n(E,{}_{\sigma\!}E)$. Since $\Ext_{E \otimes_k \opcat{E}}^n(E,{}_{\sigma\!}E) \cong \Hom_{\bbimder{E}{E}}(E, {}_{\sigma\!}E[n])$, the element $\alpha$ gives rise to a triangle
	\begin{center}
		\begin{tikzcd}[sep=small]
			Y \arrow[r] & E \arrow[r, "\alpha"] & {}_{\sigma\!}E[n] \arrow[r,rightsquigarrow] & {}
		\end{tikzcd}
	\end{center}
	in $\bbimder{E}{E}$. By construction, the functor \begin{tikzcd}[cramped, sep=small]
		Y \ltens_E - : D^b(E) \arrow[r] & D^b(E)
	\end{tikzcd} is triangulated, and induces the functor \begin{tikzcd}[cramped, sep=small]
		\overbar{Y} \otimes_E - : E\stmod \arrow[r] & E\stmod
	\end{tikzcd} so that, for any $E$-module $U$, we have $\overbar{Y} \otimes_E U \cong \overbar{Y \ltens_E U}$ in $E\stmod$.
	
	Dually, one may show that, for any $\opcat{E}$-module $V$, we have $V \otimes_E \overbar{Y} \cong \overbar{V \ltens_E Y}$ in $\opcat{E}\stmod$. This gives us the following characterisation of strongly periodic modules over $E$ and $\opcat{E}$, at the level of the derived category.
	
	\begin{lem}\label{lem:strongperiodic}
		The $E$-module $M$ is strongly $\sigma$-periodic relative to $\alpha$ if and only if $Y \ltens_E M$ is a perfect object in $D^b(E)$. Dually, the $\opcat{E}$-module $N$ is strongly $\sigma^{-1}$-periodic relative to $\alpha$ if and only if $N \ltens_E Y$ is a perfect object in $D^b(\opcat{E})$.
	\end{lem}
	
	\begin{proof}
		The triangle
		\begin{center}
			\begin{tikzcd}
				Y \ltens_E M \arrow[r,"f \ltens_E M"] & M \arrow[r,"\alpha \ltens_E M"] & {}_{\sigma\!}M[n] \arrow[r,rightsquigarrow] & {}
			\end{tikzcd}
		\end{center}
		in $D^b(E)$ induces a triangle
		\begin{center}
			\begin{tikzcd}[sep=small]
				\overbar{Y} \otimes_E M \arrow[r] & M \arrow[r] & {}\Omega_E^{-n}({}_{\sigma\!}M) \arrow[r,rightsquigarrow] & {}
			\end{tikzcd}
		\end{center}
		in $E\stmod$. If $M$ is strongly $\sigma$-periodic relative to $\alpha$, then this second arrow is an isomorphism, and thus $\overbar{Y} \otimes_E M \cong \overbar{Y \ltens_E M} \cong 0$ in $E\stmod$, so $Y \ltens_E M$ is perfect in $D^b(E)$. But on the other hand, if $Y \ltens_E M$ is perfect, then $\overbar{Y \ltens_E M} \cong \overbar{Y} \otimes_E M \cong 0$, giving an isomorphism \begin{tikzcd}[cramped,sep=small]M \arrow[r,"\sim"] & \Omega_E^{-n}({}_{\sigma\!}M)\end{tikzcd}, so that $M$ is strongly $\sigma$ periodic relative to $\alpha$. The proof of the dual statement is similar.
	\end{proof}
	
	We can use this to prove the following proposition, which tells us that strongly periodic modules occur in the Grant setting.
	
	\begin{prop}\label{prop:grantperiodicmodule}
		Suppose there is some automorphism $\sigma$ of $E$ and a subalgebra $B$ of $E$ such that $E$ is $\sigma$-periodic of period $n$ relative to $B$. Suppose that $P$ is projective as a $\opcat{B}$-module and that $P^\vee$ is projective as a $B$-module. Then there is an $\alpha \in \Ext^n_{E \otimes_k \opcat{E}}(E,{}_{\sigma\!}E)$ such that the $E$-module $M$ is strongly $\sigma$-periodic of period $n$ and the $\opcat{E}$-module $M^\vee$ is strongly $\sigma^{-1}$-periodic of period $n$, both relative to $\alpha$.
	\end{prop}
	
	\begin{proof}
		We have an exact sequence
		\begin{center}
			\begin{tikzcd}[sep=small]
				0 \arrow[r] & _{\sigma\!}E[n-1] \arrow[r] & Y \arrow[r] & E \arrow[r] & 0
			\end{tikzcd}
		\end{center}
		in $\bbimcomp{E}{E}$, where $Y$ is a truncated resolution of $E$ relative to $B$. We thus have a triangle 
		\begin{center}
			\begin{tikzcd}[sep=small]
				Y \arrow[r] & E \arrow[r, "\alpha"] & {}_{\sigma\!}E[n] \arrow[r,rightsquigarrow] & {}
			\end{tikzcd}
		\end{center}
		in $\bbimder{E}{E}$. By Lemma \ref{lem:strongperiodic}, it suffices to show that $Y \ltens_E M$ is perfect in $D^b(E)$. By construction, this is the case if and only if $(E \otimes_B E) \ltens_E M \cong E \ltens_B M$ is perfect in $D^b(E)$.
		
		We have
		\begin{align*}
			P^\vee &= \Hom_A(P,A) \\
			&\cong \Hom_A(P, P \oplus Q) \\
			&\cong \Hom_A(P,P) \oplus \Hom_A(P,Q) \\
			&= E \oplus M,
		\end{align*}
		and since $P^\vee$ is a projective $B$-module, both $E$ and $M$ are projective $B$-modules. Since $M$ is a projective $B$-module, $E \ltens_B M$ is perfect in $D^b(E)$, so that $Y \ltens_E M$ is, too, and this proves the claim for $M$.
		
		For the $\opcat{E}$-module $M^\vee$, we first recall that ${}_{\sigma\!}E \cong E_{\sigma^{-1}}$ as $E$-$E$-bimodules, and that by Theorem \ref{thm:symmetric}, since $E$ is a symmetric algebra, we have
		\[
		M^\vee = \Hom_A(P,Q)^\vee \cong \Hom_A(P,Q)^\ast \cong \Hom_A(Q,P).
		\]
		
		We similarly need to show that $M^\vee \ltens_E Y$ is perfect in $D^b(\opcat{E})$, for which it again suffices to show that $M^\vee \ltens_E (E \otimes_B E) \cong M^\vee \ltens_B E$ is. Similarly to before, we have $P \cong E \oplus M^\vee$, and since $P^\vee$ is a projective $B$-module, both $E$ and $M$ are projective $B$-modules. We can therefore deduce that $M^\vee \ltens_B E$ is perfect in $D^b(\opcat{E})$, so that $M^\vee \ltens_E Y$ is, too, and we are done.
	\end{proof}
	
	One may wonder if other examples of strongly periodic modules exist. That is, does there exist a finite-dimensional self-injective $k$-algebra $E$ and an $E$-module $M$ such that $M$ is strongly $\sigma$-periodic, but $E$ is not $\sigma$-periodic, relative or otherwise? In Section \ref{subsec:exotic}, we will see an exotic example of a strongly periodic module $M$ over a symmetric algebra $E$, which is not known to be (relatively) $\sigma$-periodic. However, it remains open to find a strongly $\sigma$-periodic module over an algebra known to not be (relatively) $\sigma$-periodic.

	\subsection{The main theorem}
	
	Recall that a Serre subcategory of $A\modcat$ is generated by a set of simple $A$-modules.
	
	\begin{defn}
		Let $P$ be a projective $A$-module. There is some subset $J \subset I$ such that $P=\bigoplus_{i \in I \setminus J}P_i^{m_i}$ for some integers $m_i \ge 1$. We call the Serre subcategory $\mathcal{A}_1$ of $A\modcat$ generated by the set $\{S_j\}_{j \in J}$ the \emph{Serre subcategory prime to $P$}.
	\end{defn}
	
	Clearly, the Serre subcategory $\mathcal{A}_1$ of $A\modcat$ prime to $P$ depends only on the isomorphism classes of summands of $P$, and not their multiplicities.
	
	The remainder of this section is dedicated to proving the following theorem.
	
	\begin{thm}\label{thm:main}
		Let $A$ be a finite-dimensional, symmetric $k$-algebra. Let $P$ and $Q$ be projective $A$-modules with no common direct summands up to isomorphism, and such that $P \oplus Q$ is a projective generator of $A$. Set $E = \opcat{\End_A(P)}$ and $M=\Hom_A(P,Q)$. Let $\mathcal{A}_1$ be the Serre subcategory of $A\modcat$ prime to $P$. Then there exists a standard derived equivalence \begin{tikzcd}[cramped, sep=small]
			\Phi: D^b(A) \arrow[r,"\sim"] & D^b(A)
		\end{tikzcd} self-perverse relative to
		\[
		0 \subset_0 \mathcal{A}_1 \subset_n A\modcat,
		\]
		together with a natural transformation \begin{tikzcd}[cramped, sep=small]\Id_{D^b(A)} \arrow[r] & \Phi\end{tikzcd} restricting to a natural isomorphism \begin{tikzcd}[cramped, sep=small]\Id_{\mathcal{A}_1} \arrow[r,"\sim"] & \Phi|_{\mathcal{A}_1}\end{tikzcd} if and only if there exists an automorphism $\sigma$ of $E$ and an extension $\alpha \in \Ext_{E\otimes_k\opcat{E}}^n(E,{}_{\sigma\!}E)$ such that the $E$-module $M$ is strongly $\sigma$-periodic and the $\opcat{E}$-module $M^\vee$ is strongly $\sigma^{-1}$-periodic, both of period $n$ and relative to $\alpha$.
	\end{thm}
	
	The statement that $P \oplus Q$ is a projective generator of $A$ means that there is some $m \in \mathbb{Z}_+$ such that the regular $A$-module $A$ is a direct summand of $(P \oplus Q)^{\oplus m}$. Further, we are assuming that $P$ and $Q$ have no common direct summands. Then there is some $J \subset I$ and integers $m_i, m_j \ge 1$ such that $P= \bigoplus_{i \in I \setminus J} P_i^{m_i}$ and $Q= \bigoplus_{j \in J} P_j^{m_j}$.
	
	We will prove Theorem \ref{thm:main} in two parts. Firstly, Theorem \ref{thm:perversetoperiodic} tells us that the existence of such a derived autoequivalence $\Phi$ guarantees that $M$ and $M^\vee$ are twisted strongly periodic. Theorem \ref{thm:periodictoperverse} demonstrates the converse: that twisted strong periodicity of the modules $M$ and $M^\vee$ guarantee the existence of a derived autoequivalence $\Phi$ with the requisite properties.

	\subsection{Necessary conditions}\label{subsec:necconds}

	Our first main result is the following.
	
	\begin{thm}\label{thm:perversetoperiodic}
		Suppose that \begin{tikzcd}[cramped, sep=small]
			\Phi: D^b(A) \arrow[r,"\sim"] & D^b(A)
		\end{tikzcd} is a standard derived autoequivalence, perverse relative to a filtration
		\[
		0 \subset_0 \mathcal{A}_1 \subset_n A\modcat
		\]
		of Serre subcategories. Suppose also that there is a natural transformation of functors $\Id_{D^b(A)} \to \Phi$ restricting to a natural isomorphism \begin{tikzcd}[cramped, sep=small]\Id_{\mathcal{A}_1} \arrow[r,"\sim"] & \Phi|_{\mathcal{A}_1}\end{tikzcd}. Then there are projective $A$-modules $P$ and $Q$, with no common direct summands, such that $\mathcal{A}_1$ is the Serre subcategory prime to $P$, $P \oplus Q$ is a projective generator of $A$ and, with $E=\opcat{\End_A(P)}$, there is an automorphism $\sigma$ of $E$ such that $E$-module $M=\Hom_A(P,Q)$ is strongly $\sigma$-periodic of period $n$, and the $\opcat{E}$-module $M^\vee$ is strongly $\sigma^{-1}$-periodic of period $n$, relative to some $\alpha \in \Ext^n_{E \otimes_k \opcat{E}}(E,{}_{\sigma\!}E)$.
	\end{thm}
	
	We first note that the standard restriction on the derived equivalence $\Phi$ is not too strong. Indeed, if we were to drop this assumption, then by Theorem \ref{thm:standardderived}, there is a standard derived equivalence \begin{tikzcd}[cramped, sep=small]
		X \ltens_A -: D^b(A) \arrow[r,"\sim"] & D^b(A)
	\end{tikzcd} agreeing with $\Phi$ on objects of $D^b(A)$. Moreover, by Proposition \ref{prop:perversestandard}, if $\Phi$ is perverse relative to the stated filtration, then $X \ltens_A -$ is, too. 
	
	It will also be prudent to investigate the condition on the natural transformation $\Id_{D^b(A)} \to \Phi$. Since $\Phi$ is standard, there is some $X \in \bbimder{A}{A}$ such that $\Phi = X \ltens_A -$. Similarly, the identity functor is such that $\Id_{D^b(A)} = A \ltens_A -$. Thus, we have a natural transformation \begin{tikzcd}[cramped, sep=small]
		A \ltens_A - \arrow[r] & X \ltens_A -,
	\end{tikzcd} which by Yoneda's Lemma must come from a morphism \begin{tikzcd}[cramped, sep=small]
		A \arrow[r,"h"] & X
	\end{tikzcd} in $\bbimder{A}{A}$. Since $\Phi$ is perverse relative to the given filtration, we have a commutative diagram
	\begin{center}
		\begin{tikzcd}
			D^b(A) \arrow[r,"\Phi{[n]}"] & D^b(A) \\
			D^b_{\mathcal{A}_1}(A) \arrow[r,"\sim"] \arrow[u,hook] & D^b_{\mathcal{A}_1}(A) \arrow[u,hook] \\
			\mathcal{A}_1 \arrow[r,"\sim"] \arrow[u,hook] & \mathcal{A}_1 \arrow[u,hook]
		\end{tikzcd}
	\end{center}
	and $\Phi$ restricts in this way to an autoequivalence $\Phi|_{\mathcal{A}_1}$ of $\mathcal{A}_1$. For every $A$-module $V \in \mathcal{A}_1$, the induced map
	\begin{center}
		\begin{tikzcd}[row sep=small]
			A \ltens_A V \arrow[r,"h \ltens_A V"] \arrow[d,equal] & X \ltens_A V \arrow[d,equal] \\
			V \arrow[r,"\sim"] & \Phi(V)
		\end{tikzcd}
	\end{center}
	is an isomorphism.
	
	Next, we note that we may assume that the projective $A$-modules $P$ and $Q$ are direct sums of projective indecomposable modules, no two of which are isomorphic.
	
	\begin{prop}\label{prop:Aisbasic1}
		Let $P$ and $Q$ be projective $A$-modules such that $P \oplus Q$ is a projective generator of $A$. Let $J \subset I$ be such that $P = \bigoplus_{i \in I \setminus J}P_i^{m_i}$ and $Q = \bigoplus_{j \in J} P_j^{m_j}$ for some integers $m_i,m_j \ge 1$. Set $P' = \bigoplus_{i \in I \setminus J}P_i$ and $Q' = \bigoplus_{j \in J} P_j$. With $E=\opcat{\End_A(P)}$ and $M= \Hom_A(P,Q)$, and $E'=\opcat{\End_A(P')}$ and $M'= \Hom_A(P',Q')$, there is an automorphism $\sigma$ of $E$ and $n \in \mathbb{Z}_+$ such that the $E$-module $M$ is strongly $\sigma$-periodic of period $n$ if and only if there is an automorphism $\sigma'$ of $E'$ such that the $E'$-module $M'$ is strongly $\sigma'$-periodic of period $n$.
	\end{prop}
	
	\begin{proof}
		Let $V = \Hom_A(P,P') \cong P^\vee \otimes_A P'$. Then $V$ is an $E$-$E'$-bimodule, projective as a left $E$-module and as a right $E'$-module. As an $E$-module, $V$ is a projective generator, so by standard Morita theory induces a Morita equivalence \begin{tikzcd}[cramped,sep=small]E\modcat \arrow[r,"\sim"] & E'\modcat. \end{tikzcd} In particular, $V \otimes_{E'} V^\vee \cong E$, and $V^\vee \otimes_E V \cong E'$.
		
		Let $\ve$ be the counit of the adjunction $P \otimes_E - \dashv P^\vee \otimes_A - $. Then, since $P \otimes_E P^\vee \otimes_A P \cong P \otimes_E E \cong P$, the map \begin{tikzcd}[cramped, sep=small]
			\ve_{P}: P \otimes_E P^\vee \otimes_A P \arrow[r] & P
		\end{tikzcd} is an isomorphism, and hence so is \begin{tikzcd}[cramped, sep=small]
			\ve_{P'}: P \otimes_E P^\vee \otimes_A P' \arrow[r] & P',
		\end{tikzcd} as $P' \in P\cadd$. Since $V=P^\vee \otimes_A P'$, we thus have $P \otimes_E V \cong P'$, and $P \cong P' \otimes_{E'} V^\vee$.
		
		Similarly, if $\eta$ is the unit of the adjunction $- \otimes_{E'} (P')^\vee \dashv - \otimes_A P$, then the map \begin{tikzcd}[cramped, sep=small]
			\eta_{P^\vee}: P^\vee \arrow[r] & P^\vee \otimes_A P' \otimes_{E'} (P')^\vee
		\end{tikzcd} is an isomorphism, since $P^\vee \in (P')^\vee\cadd$. Again, $P^\vee \otimes_A P' \cong V$, so we have $P^\vee \cong V \otimes_{E'} (P')^\vee$, and $V^\vee \otimes_E P^\vee \cong (P')^\vee$.
		
		Suppose first that there is some automorphism $\sigma$ of $E$ such that $M$ is strongly $\sigma$-periodic of period $n$, relative to $\alpha \in \Ext^n_{E\otimes_k\opcat{E}}(E,{}_{\sigma\!}E)$. By Lemma \ref{lem:strongperiodic}, there is a distinguished triangle
		\begin{center}
			\begin{tikzcd}[sep=small]
				Y \arrow[r] & E \arrow[r, "\alpha"] & {}_{\sigma\!}E[n] \arrow[r,rightsquigarrow] & {}
			\end{tikzcd}
		\end{center}
		such that $Y \ltens_E M$ is a perfect object in $D^b(E)$. Since $V = P^\vee \otimes_A P'$ is projective as a left $E$-module, the functor \[V^\vee \otimes_E - \otimes_E V: E \otimes_k \opcat{E}\modcat \to E' \otimes_k \opcat{(E')}\modcat\] is exact. Thus, we obtain a triangle
		\begin{center}
			\begin{tikzcd}[sep=small]
				Y' \arrow[r] & E' \arrow[r,"\alpha'"] & {}_{\sigma'\!}E'[n] \arrow[r,rightsquigarrow] & {},
			\end{tikzcd}
		\end{center}
		where $Y' = V^\vee \otimes_E Y \otimes_E V$, $\alpha' = V^\vee \otimes_E \alpha \otimes_E V$, and $\sigma'$ is the restriction of $\sigma$ to $E'$, noting that $V^\vee \otimes_E E \otimes_E V \cong V^\vee \otimes_E V \cong E'$. Then
		\begin{align*}
			Y' \ltens_{E'} M' &\cong (V^\vee \otimes_E Y \otimes_E V) \ltens_{E'} ((P')^\vee \otimes_A Q') \\
			&\cong V^\vee \otimes_E Y \ltens_E P^\vee \otimes_A Q',
		\end{align*}
		since $V \otimes_{E'} (P')^\vee \cong P^\vee$. Since $Q' \in Q\cadd$ and $Y \ltens_E M \cong Y \ltens_E P^\vee \otimes_A Q$ is perfect by assumption, the same is true of $Y \ltens_E P^\vee \otimes_A Q'$. Thus, since $V^\vee$ is projective as a right $E$-module, $Y'\ltens_{E'} M'$ is perfect in $D^b(E')$. By Lemma \ref{lem:strongperiodic}, $M'$ is a strongly $\sigma'$-periodic $E'$-module, relative to $\alpha'$.
		
		Conversely, suppose that there is $\alpha' \in \Ext^n_{E'\otimes_k\opcat{E'}}(E',{}_{\sigma\!}E')$ and an automorphism $\sigma'$ of $E'$ such that $M'$ is strongly $\sigma'$-periodic relative to $\alpha'$. Similarly to the above, we have a triangle
		\begin{center}
			\begin{tikzcd}[sep=small]
				Y' \arrow[r] & E' \arrow[r,"\alpha'"] & {}_{\sigma'\!}E'[n] \arrow[r,rightsquigarrow] & {},
			\end{tikzcd}
		\end{center}
		and a triangulated functor \[V \otimes_{E'} - \otimes_{E'} V: \bbimder{E'}{E'} \to \bbimder{E}{E}.\] Set $Y= V \otimes_{E'} Y' \otimes_{E'} V$. We have $E \cong V \otimes_{E'} E' \otimes_{E'} V$. We thus have a triangle
		\begin{center}
			\begin{tikzcd}[sep=small]
				Y \arrow[r] & E \arrow[r,"\alpha"] & W \arrow[r,rightsquigarrow] & {}
			\end{tikzcd}
		\end{center}
		in $\bbimder{E}{E}$, where $W = V \otimes_{E'} {}_{\sigma'\!}E'[n] \otimes_{E'} V$. It is clear that $W \cong E[n]$ in $D^b(\opcat{E})$. Thus, as an object of $\bbimder{E}{E}$, $W$ is isomorphic to the $E$-$E$-bimodule $E$ concentrated in degree $n$, with the regular right action of $E$. The left action of $E$ then passes through an algebra homomorphism $\sigma: E \to E$. But it is clear that $W \cong E[n]$ when restricted to $D^b(E)$, too, so the algebra homomorphism $\sigma$ is an automorphism. Therefore, $W \cong {}_{\sigma\!}E[n]$ in $\bbimder{E}{E}$, and we have a triangle
		\begin{center}
			\begin{tikzcd}[sep=small]
				Y \arrow[r] & E \arrow[r,"\alpha"] & {}_{\sigma\!}E[n] \arrow[r,rightsquigarrow] & {},
			\end{tikzcd}
		\end{center}
		from which, by an analogous argument to the above, it follows that $M$ is strongly $\sigma$-periodic relative to $\alpha$.
	\end{proof}
	
	The obvious dual statement for right modules is also true, in a very similar way. In particular, we may assume that $A$ is basic and that $A \cong P \oplus Q$ as $A$-modules.
	
	Finally, we require a lemma.
	
	\begin{lem}\label{lem:ZinthickP}
		Let $P$ be a projective $A$-module. Let $\mathcal{A}_1$ be the Serre subcategory of $A\modcat$ prime to $P$. Then for a perfect complex $Z \in D^b(\opcat{A})$, we have $Z \in \langle P^\vee \rangle$ if and only if $Z \ltens_A V = 0$ for all $V \in D^b_{\mathcal{A}_1}(A)$.
	\end{lem}
	
	We comment that the $V$ in Lemma \ref{lem:ZinthickP} is different from the $V$ in Proposition \ref{prop:Aisbasic1}.
	
	\begin{proof}
		First, note that for $V \in D^b_{\mathcal{A}_1}(A)$, since $\mathcal{A}_1$ is prime to $P$, we have $P^\vee \ltens_A V = \Hom_A(P,V) \cong 0$, so one direction is clear. 
		
		For the other, suppose that $Z \ltens_A V = 0$ for all $V \in D^b_{\mathcal{A}_1}(A)$. It incurs no loss of generality to assume that $Z$ is a bounded below complex of projective $\opcat{A}$-modules, with $Z_0 \neq 0$ and $Z_m = 0$ for all $m < 0$. Assume $Z$ is such that the maximum non-zero degree $N_Z = \max\{m \ge 0: Z_m \ne 0\}$ is minimal among objects of $D^b(\opcat{A})$ with these properties. Since $Z$ is perfect, $N_Z$ is finite. 
		
		Suppose for a contradiction that $Z \not\in \langle P^\vee \rangle$. By assumption, we have $0 = Z \ltens_A V \cong \rhom_A(Z^\vee,V)$, for every $V \in D^b_{\mathcal{A}_1}(A)$, and for every $t \in \mathbb{Z}$, $H_t(\rhom_A(Z^\vee,V)) = \Hom_{D^b(A)}(Z^\vee, V[t]) = 0$. Let $X = Z^\vee \in D^b(A)$. Suppose that $X_0 \in P\cadd$. Then there is a commutative diagram
		\begin{center}
			\begin{tikzcd}[sep=small] 
				& & & 0 \arrow[r] & X_0 \arrow[d] \arrow[r] & 0 \\
				0 \arrow[r] & X_{N_Z} \arrow[r] \arrow[d] & \ldots \arrow[r] & X_1 \arrow[r] \arrow[d] & X_0 \arrow[r] & 0 \\ 
				0 \arrow[r] & X_{N_Z} \arrow[r] & \ldots \arrow[r]  & X_1 \arrow[r] & 0
			\end{tikzcd}
		\end{center}
		giving rise to a triangle 
		\begin{center}
			\begin{tikzcd}[sep=small]
				X \arrow[r] & X' \arrow[r] & X_0[1] \arrow[r,rightsquigarrow] & {}
			\end{tikzcd}
		\end{center}
		in $D^b(A)$, with $X'$ the complex defined by the third row of the diagram. Since $X_0 \in P\cadd$ and by the assumption on $Z$, we must have $\rhom_A(X',V) \cong 0$ for every $V \in D^b_{\mathcal{A}_1}(A)$. This contradicts the minimality of $Z$, since $(X')^\vee$ has strictly smaller maximum non-zero degree. Therefore $X_0 \not\in P\cadd$. 
		
		Then, with $X_0 \not\in P\cadd$, there is some simple module $S$ in $D^b_{\mathcal{A}_1}(A)$ such that $S$ is a summand of $X_0/\rad(X_0)$. But then the morphism
		\begin{center}
			\begin{tikzcd}[sep=small]
				0 \arrow[r] & X_{N_Z} \arrow[r] & \ldots \arrow[r] & X_1 \arrow[r] & X_0 \arrow[r] \arrow[d] & 0 \\
				& & & 0 \arrow[r] & S \arrow[r] & 0
			\end{tikzcd}
		\end{center}
		is a non-zero element of $\Hom_{D^b(A)}(X,S) = \Hom_{D^b(A)}(Z^\vee,S)$. This is a contradiction. Hence, $Z \in \langle P^\vee \rangle$.
	\end{proof}
	
	We now have all the tools to prove Theorem \ref{thm:perversetoperiodic}.
	
	\begin{proof}[Proof of Theorem \ref{thm:perversetoperiodic}]
		We will assume that $A$ is basic. By Proposition \ref{prop:Aisbasic1}, this is a legitimate assumption.
		
		Let $J \subset I$ be such that set $\{S_j\}_{j \in J}$ generates $\mathcal{A}_1$. Set $P= \oplus_{i \in I \setminus J} P_i$ and $Q = \oplus_{j \in J} P_j$. Then by construction, $\mathcal{A}_1$ is the Serre subcategory prime to $P$, $P$ and $Q$ have no common direct summands, and $A \cong P \oplus Q$ as $A$-modules.
		
		Since $\Phi$ is standard, there is some $X \in \bbimder{A}{A}$ such that $\Phi=X \ltens_A -$. By assumption, we have a map \begin{tikzcd}[cramped, sep=small]A \arrow[r,"h"] & X\end{tikzcd} in $\bbimder{A}{A}$. This gives rise to a triangle
		\begin{center}
			\begin{tikzcd}[sep=small]
				Z \arrow[r] & A \arrow[r,"h"] & X \arrow[r,rightsquigarrow] & {}
			\end{tikzcd}
		\end{center}
		in $\bbimder{A}{A}$, say $\Delta$. Applying the triangulated functor
		\begin{center}
			\begin{tikzcd}[sep=small]
				\rhom_{A\text{-}A}(P \otimes_k P^\vee, -): \bbimder{A}{A} \arrow[r] & \bbimder{E}{E}
			\end{tikzcd}
		\end{center}
		to $\Delta$, we obtain a triangle
		\begin{center}
			\begin{tikzcd}[sep=small]
				P^\vee \ltens_A Z \ltens_A P \arrow[r] & P^\vee \ltens_A A \ltens_A P \arrow[r] & P^\vee \ltens_A X \ltens_A P \arrow[r,rightsquigarrow] & {}
			\end{tikzcd}
		\end{center}
		in $\bbimder{E}{E}$.
		
		First, we have $P^\vee \ltens_A A \ltens_A P \cong E$. Next, we note that by Proposition \ref{prop:rightmodperverse}, the equivalence
		\begin{center}
			\begin{tikzcd}[sep=small]
				- \ltens_A X: D^b(\opcat{A}) \arrow[r,"\sim"] & D^b(\opcat{A})
			\end{tikzcd}
		\end{center}
		is also a perverse equivalence, relative to 
		\[
		0 \subset_0 \mathcal{A}'_1 \subset_n \opcat{A}\modcat,
		\]
		where $\mathcal{A}'_1$ is the Serre subcategory of $\opcat{A}\modcat$ prime to $P^\vee$. In particular, we have $P^\vee \ltens_A X \cong P^\vee[n]$ as an object of $D^b(\opcat{A})$. We thus have that $P^\vee \ltens_A X \ltens_A P \cong P^\vee[n] \ltens_A P \cong E[n]$ in $D^b(\opcat{E})$. That is, in $\bbimder{E}{E}$, $P^\vee \ltens_A X \ltens_A P$ is isomorphic to the $E$-$E$-bimodule $E$ concentrated in degree $n$, with the regular right action of $E$. The left action of $E$ on this $E$-$E$-bimodule must therefore pass through some algebra homomorphism \begin{tikzcd}[cramped,sep=small]\sigma: E \arrow[r] & E.\end{tikzcd} But note that $X \ltens_A P \cong P[n]$ in $D^b(A)$, so that $P^\vee \ltens_A X \ltens_A P \cong P^\vee \ltens_A P[n] \cong E[n]$ in $D^b(E)$, too. Thus, the homomorphism $\sigma$ must be an isomorphism. In other words, $P^\vee \ltens_A X \ltens_A P \cong {}_{\sigma\!}E[n]$ in $\bbimder{E}{E}$.
		
		Setting $Y = P^\vee \ltens_A Z \ltens_A P$, we therefore have a triangle
		\begin{center}
			\begin{tikzcd}[sep=small]
				Y \arrow[r] & E \arrow[r,"\alpha"] & {}_{\sigma\!}E[n] \arrow[r,rightsquigarrow] & {}
			\end{tikzcd}
		\end{center}
		in $\bbimder{E}{E}$, say $\nabla$. This defines an element $\alpha \in \Ext^n_{E\text{-}E}(E, {}_{\sigma\!}E[n])$. By Lemma \ref{lem:strongperiodic}, to prove that $M$ is strongly $\sigma$-periodic relative to $\alpha$, it suffices to show that $Y \ltens_E M$ is a perfect object in $D^b(E)$.
		
		To this end, take the object $Y \ltens_E M$ of $D^b(E)$. We have $Y \ltens_E M \cong P^\vee \ltens_A Z \ltens_A P \ltens_E P^\vee \otimes_A Q$. Consider the adjunction $- \ltens_A P \dashv - \ltens_E P^\vee$. For any object $W$ in the thick subcategory $\langle P^\vee \rangle$ of $D^b(\opcat{A})$, we have $W \ltens_A P \ltens_E P^\vee \cong W$ in $D^b(\opcat{A})$. Thus, if we can show that $P^\vee \ltens_A Z \in \langle P^\vee \rangle$, then we will have $P^\vee \ltens_A Z \ltens_A P \ltens_E P^\vee \cong P^\vee \ltens_A Z$, so that $Y \ltens_E M \cong P^\vee \ltens_A Z \ltens_A Q$. Since $P^\vee \ltens_A Z$ is a summand of $A \ltens_A Z$, as $P^\vee$ is projective, we need only show that $Z \in \langle P^\vee \rangle$, considered as an object of $D^b(\opcat{A})$. By Lemma \ref{lem:ZinthickP}, this is equivalent to showing that $Z \ltens_A V = 0$ for all $V \in D^b_{\mathcal{A}_1}(A)$. 
		
		Given $V \in D^b_{\mathcal{A}_1}(A)$, we have a triangle $\Delta \ltens_A V$,
		\begin{center}
			\begin{tikzcd}
				Z \ltens_A V \arrow[r] & A \ltens_A V \arrow[r,"h \ltens_A V"] & X \ltens_A V \arrow[r,rightsquigarrow] & {}.
			\end{tikzcd}
		\end{center}
		By assumption, $h \ltens_A V$ is an isomorphism, so $Z \ltens_A V = 0$. Thus, $Y \ltens_E M \cong P^\vee \ltens_A Z \ltens_A Q$.
		
		By Lemma \ref{lem:strongperiodic}, it suffices to show that this right hand side is perfect in $D^b(E)$. The object $P^\vee \ltens_A Z$ fits into a triangle
		\begin{center}
			\begin{tikzcd}[sep=small]
				P^\vee \ltens_A Z \arrow[r] & P^\vee \ltens_A A \arrow[r] & P^\vee \ltens_A X \arrow[r,rightsquigarrow] & {}.
			\end{tikzcd}
		\end{center}
		Clearly, $P^\vee \ltens_A A \cong P^\vee$ is a perfect complex of $\opcat{A}$-modules. By assumption, $X$ is perfect in $D^b(A)$ and in $D^b(\opcat{A})$. Thus, since $P^\vee$ is a projective $\opcat{A}$-module, $P^\vee \ltens_A X$ is perfect in $D^b(\opcat{A})$, so $P^\vee \ltens_A Z$ is, too. Then, since $Q$ is a projective $A$-module, the object $P^\vee \ltens_A Z \ltens_A Q \cong Y \ltens_E M$ is perfect in $D^b(E)$. This completes the proof of the claim for $M$.
		
		For the claim on $M^\vee$, it suffices to show that $M^\vee \ltens_E Y \cong Q^\vee \otimes_A P \ltens_E P^\vee \ltens_A Z \ltens_A P$ is perfect in $D^b(\opcat{E})$. A similar argument to the above, using the obvious dual statement to Lemma \ref{lem:ZinthickP}, will show that $Z \ltens_A P \in \langle P \rangle$, so that from the adjunction $P^\vee \ltens_A - \dashv P \ltens_E -$ we have $M^\vee \ltens_E Y \cong Q^\vee \otimes_A P \ltens_E P^\vee \ltens_A Z \ltens_A P \cong Q^\vee \ltens_A Z \ltens_A P$. Then, the triangle
		\begin{center}
			\begin{tikzcd}[sep=small]
				Z \ltens_A P \arrow[r] & A \ltens_A P \arrow[r] & X \ltens_A P \arrow[r,rightsquigarrow] & {}
			\end{tikzcd}
		\end{center}
		guarantees that $Z \ltens_A P$ is perfect in $D^b(A)$, so that, since $Q$ is projective and $A$ is symmetric, $M^\vee \ltens_E Y \cong Q^\vee \ltens_A Z \ltens_A P$ is perfect in $D^b(\opcat{E})$.
	\end{proof}
	
	We thus have the first part of Theorem \ref{thm:main}.

	\subsection{Sufficient conditions}\label{subsec:suffconds}

	The next step is to show that the converse to Theorem \ref{thm:perversetoperiodic} also holds true. That is, twisted strongly periodic $E$-modules of period $n$ give rise to two-step self-perverse equivalences of width $n$ of the appropriate form.
	
	\begin{thm}\label{thm:periodictoperverse}
		Let $P$ and $Q$ be projective $A$-modules with no common direct summands such that $P \oplus Q$ is a projective generator of $A$. Let $E=\opcat{\End_A(P)}$ and $M=\Hom_A(P,Q)$. If there is an automorphism $\sigma$ of $E$, $\alpha \in \Ext^n_{E \otimes_k \opcat{E}}(E,{}_{\sigma\!}E)$ and $n \in \mathbb{Z}_+$ such that the $E$-module $M$ is strongly $\sigma$-periodic and the $\opcat{E}$-module $M^\vee$ is strongly $\sigma^{-1}$-periodic, both of period $n$ and relative to $\alpha$, then there is a standard derived equivalence \begin{tikzcd}[cramped, sep=small]
			\Phi_P: D^b(A) \arrow[r,"\sim"] & D^b(A),
		\end{tikzcd} perverse relative to
		\[
		0 \subset \mathcal{A}_1 \subset A\modcat,
		\] 
		where $\mathcal{A}_1$ is the Serre subcategory of $A\modcat$ prime to $P$.
	\end{thm}
	
	Our proof of Theorem \ref{thm:periodictoperverse} adapts the method of Grant \cite{grant_periodic} to work directly in the derived category, itself based on Ploog's simplified proof that (geometric) spherical twists are derived autoequivalences \cite{ploog_thesis}. We require the following definitions.
	
	\begin{defn}
		Let $\mathcal{S}$ be a collection of objects in a triangulated category $\mathcal{T}$. The \emph{right orthogonal complement} of $\mathcal{S}$ is
		\[
		\mathcal{S}^\perp = \{V \in \mathcal{T}: \Hom_{\mathcal{T}}(U,V[i]) = 0 \text{ for all } U \in \mathcal{S}, i \in \mathbb{Z}\}.
		\]
		The \emph{left orthogonal} complement $^{\perp\!}\mathcal{S}$ of $\mathcal{S}$ is defined similarly.
	\end{defn}
	
	By \cite[Corollary 3.2]{rickard_symmderived}, if $Z$ is a bounded complex of projective $A$-modules and $V$ is any object of $D^b(A)$, since $A$ is symmetric, $\Hom_{D^b(A)}(Z,V)$ and $\Hom_{D^b(A)}(V,Z)$ are naturally dual as $k$-vector spaces\footnote{That is, $Z$ is a 0-Calabi-Yau object in $D^b(A)$}. In such instances, the right and left orthogonal complements $Z^\perp$ and $^{\perp\!}Z$ coincide, and we may refer unambiguously to the \emph{orthogonal complement} $Z^\perp$ of $Z$. In particular, if $P$ is a projective $A$-module, the orthogonal complement $P^\perp$ is unambiguously defined. We comment that $P^\perp = D^b_{\mathcal{A}_1}(A)$, while the proof of Lemma \ref{lem:ZinthickP} shows that ${}^{\perp\!}D^b_{\mathcal{A}_1}(A) = \Perf(A) \cap \langle P \rangle$.
	
	\begin{defn}
		A collection of objects $\mathcal{S}$ in a triangulated category $\mathcal{T}$ is a \emph{spanning class} for $\mathcal{T}$ if for every $V \in \mathcal{T}$, if $\Hom_{\mathcal{T}}(U,V[i]) = 0$ for every $U \in \mathcal{S}$ and all $i \in \mathbb{Z}$, then $V \cong 0$, and if $\Hom_{\mathcal{T}}(V[i],U) = 0$ for every $U \in \mathcal{S}$ and all $i \in \mathbb{Z}$, then $V \cong 0$.
	\end{defn}
	
	The following lemma is \cite[Lemma 3.14]{grant_periodic}.
	
	\begin{lem}\label{lem:spanningclass}
		If $P$ is a projective $A$-module, then the collection of objects $\mathcal{S} = \{P\} \cup P^\perp$ is a spanning class for $D^b(A)$.
	\end{lem}
	
	Suppose the conditions of Theorem \ref{thm:periodictoperverse} hold. We now identify our functor \begin{tikzcd}[cramped, sep=small]\Phi: D^b(A) \arrow[r] & D^b(A).\end{tikzcd}
	
	The extension $\alpha \in \Ext_{E \otimes_k \opcat{E}}(E, {}_{\sigma\!}E)$ gives rise to a triangle
	\begin{center}
		\begin{tikzcd}[sep=small]
			Y \arrow[r,"f"] & E \arrow[r,"\alpha"] & {}_{\sigma\!}E[n] \arrow[r,rightsquigarrow] & {}
		\end{tikzcd}
	\end{center}
	in $\bbimder{E}{E}$, say $\nabla$. We have a chain of isomorphims
	\begin{align*}
		\Hom_{\bbimder{E}{E}}(Y,E) &\cong \Hom_{\bbimder{E}{E}}(Y, \rhom_A(P,P)) \\
		&\cong \Hom_{\bbimder{A}{E}}(P \ltens_E Y, P) \\
		&\cong \Hom_{\bbimder{A}{E}}(P \ltens_E Y, \rhom_{\opcat{A}}(P^\vee, A)) \\
		&\cong \Hom_{\bbimder{A}{A}}(P \ltens_E Y \ltens_E P^\vee, A)
	\end{align*}
	given by tensor-Hom adjunction. Let $g: P \ltens_E Y \ltens_E P^\vee \to A$ be the image of $f:Y \to E$ under this chain of isomorphisms. As in \cite[Lemma 3.4]{grant_periodic} we can characterise the map $g$ as the resulting map in the commutative diagram
	\begin{center}
		\begin{tikzcd}
			P \ltens_E Y \ltens_E P^\vee \arrow[d, "P \ltens_E f \ltens_E P^\vee"] \arrow[r,dashed,"g"] & A \\
			P \ltens_E E \ltens_E P^\vee \arrow[r,"\sim"] & P \ltens_E P^\vee \arrow[u,"\ve^R_A"]
		\end{tikzcd}
	\end{center}
	where $\ve^R$ is the counit of the adjunction $- \ltens_E P^\vee \dashv - \ltens_A P$. We note that $\ve^R_A$ is the usual evaluation map $P \ltens_E P^\vee \to A$. This in turn gives rise to a triangle
	\begin{center}
		\begin{tikzcd}[sep=small]
			P \ltens_E Y \ltens_E P^\vee \arrow[r,"g"] & A \arrow[r] & X \arrow[r, rightsquigarrow] & {}
		\end{tikzcd}
	\end{center}
	in $\bbimder{A}{A}$, say $\Delta$.
	
	\begin{defn}
		The functor \begin{tikzcd}[cramped, sep=small]
			\Phi_P = X \ltens_A -: D^b(A) \arrow[r] & D^b(A)
		\end{tikzcd} is the \emph{generalised periodic twist} of $A$ at $P$.
	\end{defn}
	
	Our task is to show that the generalised periodic twist $\Phi_P$ is an equivalence. We first show that we may again reduce to the case that $P$ and $Q$ are direct sums of projective indecomposable modules, no two of which are isomorphic.
	
	\begin{prop}\label{prop:Aisbasic2}
		Let $J \subset I$ such that $P = \bigoplus_{i \in I \setminus J} P_i^{m_i}$ and $Q = \bigoplus_{j \in J} P_j^{m_j}$ for integers $m_j, m_i \ge 1$. Let $P' = \bigoplus_{i \in I \setminus J} P_i$ and $Q' = \bigoplus_{j \in J} P_j$, $E'= \opcat{\End_A(P')}$ and $M'=\Hom_A(P',Q')$. Then the generalised periodic twists of $A$ at $P$ and at $P'$ coincide, $\Phi_P \cong \Phi_{P'}$.
	\end{prop}
	
	\begin{proof}
		By construction, $P'$ and $Q'$ have no common direct summands and $P' \oplus Q'$ is a projective generator of $A$. By Proposition \ref{prop:Aisbasic1}, with $E' = \opcat{\End_A(P')}$ and $M'=\Hom_A(P',Q')$, the $E'$-module $M'$ and the $\opcat{(E')}$-module $(M')^\vee$ are strongly $\sigma'$-periodic relative to $\alpha'$, where $\sigma'$ and $\alpha'$ are the restrictions of $\sigma$ and $\alpha$ respectively to $E'$. The generalised periodic twist $\Phi_{P'}$ therefore exists as constructed.
		
		Recall from the proof of Proposition \ref{prop:Aisbasic1}, there is an $E$-$E'$-bimodule $V$ such that $P' \otimes_{E'} V^\vee \cong P$ and $V \otimes_{E'} (P')^\vee \cong P^\vee$, and applying the functor $V^\vee \ltens_E - \ltens_E V$ to $\nabla$, we have a triangle
		\begin{center}
			\begin{tikzcd}[sep=small]
				Y' \arrow[r,"f"] & E' \arrow[r,"\alpha'"] & {}_{\sigma'\!}E'[n] \arrow[r,rightsquigarrow] & {}
			\end{tikzcd}
		\end{center}
		in $\bbimder{E'}{E'}$, with $Y' \cong V^\vee \ltens_E Y \ltens_E V$. Observe then that 
		\begin{align*}
			P' \ltens_{E'} Y' \ltens_{E'} (P')^\vee &\cong P' \ltens_{E'} (V^\vee \otimes_E Y \otimes_E V) \ltens_{E'} (P')^\vee \\
			& \cong P \ltens_E Y \ltens_E P^\vee.
		\end{align*}
		By the completion axiom for triangulated categories, we have a morphism of triangles
		\begin{center}
			\begin{tikzcd}[column sep=small]
				P' \ltens_{E'} Y' \ltens_{E'} (P')^\vee \arrow[r] \arrow[d,"\sim"] & A \arrow[r] \arrow[d,equal] & X' \arrow[r,rightsquigarrow] \arrow[d,dashed] & {} \\
				P \ltens_E Y \ltens_E P^\vee \arrow[r] & A \arrow[r] & X \arrow[r,rightsquigarrow] & {}
			\end{tikzcd}
		\end{center}
		in $\bbimder{A}{A}$, and by the 5-Lemma for triangulated categories, this third arrow is an isomorphism. Thus, $\Phi_P = X \ltens_A - \cong X' \ltens_A - = \Phi_{P'}$, and this completes the proof.
	\end{proof}
	
	In particular, we may assume that $A$ is basic and that $A \cong P \oplus Q$ as $A$-modules.
	
	We now work towards demonstrating that $\Phi_P$ is an equivalence. Recall that a functor is an equivalence if and only if it is fully faithful and essentially surjective. The following theorem of Bridgeland \cite[Theorem 2.3]{bridgeland_triangulated} will be useful.
	
	\begin{thm}\label{thm:bridgeland}
		Let $\mathcal{T}$, $\mathcal{T}'$ be triangulated categories and \begin{tikzcd}[cramped, sep=small]F:\mathcal{T} \arrow[r] & \mathcal{T}'\end{tikzcd} a triangulated functor with a left and a right adjoint. Then $F$ is fully faithful if and only if there is a spanning class $\mathcal{S}$ for $\mathcal{T}$ such that the homomorphisms
		\begin{center}
			\begin{tikzcd}[sep=small]
				\Hom_{\mathcal{T}}(U,V[i]) \arrow[r] & \Hom_{\mathcal{T}'}(F(U),F(V[i]))
			\end{tikzcd}
		\end{center}
		are bijective for every $U,V$ in $\mathcal{S}$ and $i \in \mathbb{Z}$.
	\end{thm}
	
	Our functor $\Phi_P$ satisfies the first clause of this theorem. 
	
	\begin{lem}\label{lem:Xisperfect}
		The object $X$ is perfect in $D^b(A)$ and $D^b(\opcat{A})$.
	\end{lem}
	
	\begin{proof}
		Consider the triangle $\Delta$,
		\begin{center}
			\begin{tikzcd}[sep=small]
				P \ltens_E Y \ltens_E P^\vee \arrow[r] & A \arrow[r] & X \arrow[r, rightsquigarrow] & {}.
			\end{tikzcd}
		\end{center}
		The $A$-$A$-bimodule $A$ is projective as an $A$-module and as an $\opcat{A}$-module. We have that
		\[
		P^\vee = \Hom_A(P,A) \cong \Hom_A(P,P\oplus Q) \cong E \oplus M
		\]
		as an $E$-module, so that $Y \ltens_E P^\vee \cong Y \oplus Y \ltens_E M$. Similarly, 
		\[
		P \cong \Hom_A(A,P) \cong \Hom_A(P \oplus Q, P) \cong E \oplus M^\vee,
		\]
		so that $P \ltens_E Y \cong Y \oplus M^\vee \ltens_E Y$. By assumption and Lemma \ref{lem:strongperiodic}, $Y \ltens_E M$ is perfect in $D^b(E)$, and $M^\vee \ltens_E Y$ is perfect in $D^b(\opcat{E})$. The triangle $\nabla$,
		\begin{center}
			\begin{tikzcd}[sep=small]
				Y \arrow[r] & E \arrow[r] & {}_{\sigma\!}E[n] \arrow[r,rightsquigarrow] & {}
			\end{tikzcd}
		\end{center}
		informs us that $Y$ is a perfect object in $D^b(E)$ and $D^b(\opcat{E})$. Then $P \ltens_E Y \ltens_E P^\vee \cong P \ltens_E Y \oplus P \ltens_E Y \ltens_E M$ in $D^b(A)$. Since $P$ is a projective $A$-module, the functor \begin{tikzcd}[cramped, sep=small]
			P \ltens_E - : D^b(E) \arrow[r] & D^b(A)
		\end{tikzcd} sends perfect objects to perfect objects, so since $Y$ and $Y \ltens_E M$ are perfect objects in $D^b(E)$, $P \ltens_E Y \ltens_E P^\vee$ is perfect in $D^b(A)$. Similarly, since $P^\vee$ is a projective $\opcat{A}$-module, $P \ltens_E Y \ltens_E P^\vee \cong Y \ltens_E P^\vee \oplus M^\vee \ltens_E Y \ltens_E P^\vee$, and $Y$ and $M^\vee \ltens_E Y$ are perfect objects of $D^b(\opcat{E})$, $P \ltens_E Y \ltens_E P^\vee$ is perfect in $D^b(\opcat{A})$. Thus, $X$ fits into the triangle $\Delta$ with two objects perfect in $D^b(A)$ and in $D^b(\opcat{A})$, so $X$ must be, too.
	\end{proof}
	
	Thus, the functor \begin{tikzcd}[cramped, sep=small]
		X^\vee \ltens_A -: D^b(A) \arrow[r] & D^b(A)
	\end{tikzcd} is both left and right adjoint to $\Phi_P$. In order to apply Theorem \ref{thm:bridgeland}, we now investigate how $\Phi_P$ acts on the spanning class $\mathcal{S} = \{P\} \cup P^\perp$.
	
	\begin{prop}\label{prop:psionperp}
		For any $V$ in $P^\perp$, $\Phi_P(V) \cong V$.
	\end{prop}
	
	\begin{proof}
		Consider the triangle $\Delta \ltens_A V$ in $D^b(A)$,
		\begin{center}
			\begin{tikzcd}[sep=small]
				P \ltens_E Y \ltens_E P^\vee \ltens_A V \arrow[r] & A \ltens_A V \arrow[r] & X \ltens_A V \arrow[r, rightsquigarrow] & {}.
			\end{tikzcd}
		\end{center}
		Clearly, $A\ltens_A V \cong V$. We have $P^\vee \ltens_A V \cong \Hom_A(P,V)$; as $P$ is a projective $A$-module, we need not derive these functors. Also since $P$ is projective, we have $\Hom_{K^b(A)}(P,V) \cong \Hom_{D^b(A)}(P,V)$, so that the homology of the complex $\Hom_A(P,V)$ is given by
		\[
		H_i(\Hom_A(P,V)) \cong \Hom_{K^b(A)}(P,V[i]) \cong \Hom_{D^b(A)}(P,V[i])
		\]
		for every $i \in \mathbb{Z}$. But $V \in P^\perp$, so $H_i(\Hom_A(P,V)) = 0$ for every $i$. Thus, $P^\vee \ltens_A V \cong \Hom_A(P,V) \cong 0$ in $D^b(A)$. The triangle $\Delta \ltens_A V$ is therefore isomorphic to the triangle
		\begin{center}
			\begin{tikzcd}[sep=small]
				0 \arrow[r] & V \arrow[r] & X \ltens_A V \arrow[r, rightsquigarrow] & {}.
			\end{tikzcd}
		\end{center}
		Thus, $\Phi_P(V) = X \ltens_A V \cong V$. 
	\end{proof}
	
	\begin{prop}\label{prop:psionP}
		We have $\Phi_P(P) \cong P[n]$.
	\end{prop}
	
	\begin{proof}
		Consider the triangles $P \ltens_E \nabla$ and $\Delta \ltens_A P$ in $D^b(A)$. The former is
		\begin{center}
			\begin{tikzcd}[sep=small]
				P \ltens_E Y \arrow[r] & P \ltens_E E \arrow[r] & P \ltens_E E_{\sigma^{-1}}[n] \arrow[r, rightsquigarrow] & {}
			\end{tikzcd}
		\end{center}
		since ${}_{\sigma\!}E \cong E_{\sigma^{-1}}$, and the latter
		\begin{center}
			\begin{tikzcd}[sep=small]
				P \ltens_E Y \ltens_E P^\vee \ltens_A P \arrow[r] & A \ltens_A P \arrow[r] & X \ltens_A P \arrow[r, rightsquigarrow] & {}.
			\end{tikzcd}
		\end{center}
		Observe first that, forgetting the right module structure, $P \ltens_E E_{\sigma^{-1}}[n] \cong P[n]$ in $D^b(A)$. We wish to build a commutative diagram
		\begin{center}
			\begin{tikzcd}
				P \ltens_E Y \arrow[r,"P \ltens_E f"] \arrow[d,"\gamma"] & P \ltens_E E \arrow[d,"\delta"]  \\
				P \ltens_E Y \ltens_E P^\vee \ltens_A P \arrow[r,"g \ltens_A P"] & A \ltens_A P 
			\end{tikzcd},
		\end{center}
		in which the vertical arrows are isomorphisms, from which the completion axiom and the 5-Lemma give an isomorphism of triangles
		\begin{center}
			\begin{tikzcd}[column sep=small]
				P \ltens_E Y \arrow[r] \arrow[d,"\sim"] & P \ltens_E E \arrow[r] \arrow[d,"\sim"] & P \ltens_E E_{\sigma^{-1}}[n] \arrow[d,"\sim"] \arrow[r, rightsquigarrow] & {} \\
				P \ltens_E Y \ltens_E P^\vee \ltens_A P \arrow[r] & A \ltens_A P \arrow[r] & X \ltens_A P \arrow[r, rightsquigarrow] & {}
			\end{tikzcd}
		\end{center}
		so that $\Phi_P(P) = X \ltens_A P \cong P \ltens_E E_{\sigma^{-1}}[n] \cong P[n]$.
		
		Let $\delta$ be the obvious isomorphism induced by the isomorphisms $P \ltens_E E \cong P \cong A \ltens_A P$. Consider the adjunction $- \ltens_E P^\vee \dashv - \ltens_A P$. Let $\ve^R$ and $\eta^R$ be the counit and unit of this adjunction. Define $\gamma$ by 
		\begin{center}
			\begin{tikzcd}
				P \ltens_E Y \arrow[r,"\eta^R_{P \ltens_E Y}"] & P \ltens_E Y \ltens_E P^\vee \ltens_A P.
			\end{tikzcd}
		\end{center}
		The triangle $P \ltens_E \nabla$ informs us that $P \ltens_E Y \in \langle P \rangle$ in $D^b(\opcat{E})$, so $\gamma$ is an isomorphism. 
		
		It thus remains to show that $(g \ltens_A P) \circ \gamma = \delta \circ (P \ltens_E f)$. From the construction of $g$, we have a commutative diagram
		\begin{center}
			\begin{tikzcd}[column sep=large]
				P \ltens_E Y \ltens_E P^\vee \ltens_A P \arrow[d, "P \ltens_E f \ltens_E P^\vee \ltens_A P"] \arrow[r,"g \ltens_A P"] & A \otimes P \\
				P \ltens_E E \ltens_E P^\vee \ltens_A P \arrow[r,"\delta \ltens_E P^\vee \ltens_A P"] & A \ltens_A P \ltens_E P^\vee \ltens_A P \arrow[u,"\ve^R_A \ltens_A P"] 
			\end{tikzcd}
		\end{center}
		so that $(g \ltens_A P) \circ \gamma = (\ve^R_A \ltens_A P) \circ (\delta \ltens_E P^\vee \ltens_A P) \circ (P \ltens_E f \ltens_E P^\vee \ltens_A P) \circ \eta^R_{P \ltens_E Y}$. Since $P^\vee \ltens_A P \cong E$, by the naturality of $\eta^R$ we have that $(g \ltens_A P) \circ \gamma = (\ve^R_A \ltens_A P) \circ \eta^R_{A \ltens_A P} \circ \delta \circ (P \ltens_E f)$, and since $(\ve^R_A \ltens_A P) \circ \eta^R_{A \ltens_A P} = \id_{A \ltens_A P}$, we have $(g \ltens_A P) \circ \gamma = \delta \circ (P \ltens_E f)$. The desired commutative diagram therefore exists, and the result follows.
	\end{proof}
	
	Combining Lemma \ref{lem:Xisperfect}, Propositions \ref{prop:psionperp}, \ref{prop:psionP} and Theorem \ref{thm:bridgeland}, we conclude the following.
	
	\begin{cor}\label{cor:Phifullyfaithful}
		The functor \begin{tikzcd}[cramped,sep=small]\Phi_P: D^b(A) \arrow[r] & D^b(A)\end{tikzcd} is fully faithful.
	\end{cor}
	
	We can now show that $\Phi$ is an equivalence.
	
	\begin{prop}\label{prop:Phiequiv}
		The functor \begin{tikzcd}[cramped,sep=small]\Phi_P: D^b(A) \arrow[r] & D^b(A)\end{tikzcd} is an equivalence.
	\end{prop}
	
	\begin{proof}
		First, we note that, since $X$ is perfect in $D^b(A)$ and in $D^b(\opcat{A})$ by Lemma \ref{lem:Xisperfect}, $\Phi_P$ restricts to a functor \begin{tikzcd}[cramped, sep=small]
			\hat{\Phi}_P: \Perf(A) \arrow[r] & \Perf(A).
		\end{tikzcd} Moreover, since $\Phi_P$ is fully faithful by Corollary \ref{cor:Phifullyfaithful}, this restriction is fully faithful, too. The image $\Phi_{P}(\Perf(A))$ is therefore a thick subcategory of $\Perf(A)$.
		
		By Proposition \ref{prop:psionP}, $\Phi_P(P) \cong P[n]$. Thus, $\langle P \rangle$ is contained in the image $\Phi_P(\Perf(A))$. Applying the functor $- \ltens_A Q$ to the triangle $\Delta$, we obtain a triangle
		\begin{center}
			\begin{tikzcd}[sep=small]
				P \ltens_E Y \ltens_E P^\vee \ltens_A Q \arrow[r] & A \ltens_A Q \arrow[r] & X \ltens_A Q \arrow[r, rightsquigarrow] & {}
			\end{tikzcd}
		\end{center}
		in $D^b(A)$. We have $P^\vee \ltens_A Q \cong M$, and since $Y \ltens_E M$ is perfect in $D^b(E)$, $P \ltens_E Y \ltens_E P^\vee \ltens_A Q \cong P \ltens_E Y \ltens_E M$ is isomorphic to an object in $\langle P \rangle$, and is therefore in $\Phi_P(\Perf(A))$. Since $Q$ is projective, $X \ltens_A Q \cong \Phi_P(Q) \in \Phi_P(\Perf(A))$. Since $\Phi_P(\Perf(A))$ is closed under triangles, $A \ltens_A Q \cong Q \in \Phi_P(\Perf(A))$. Thus, $A \cong P \oplus Q \in \Phi_P(\Perf(A))$, so $\Phi_P(\Perf(A))$ must contain all of $\Perf(A)$. In particular, the restriction of $\Phi_P$ to $\Perf(A)$ is essentially surjective, so is an equivalence \begin{tikzcd}[cramped, sep=small]
			\Phi_P: \Perf(A) \arrow[r,"\sim"] & \Perf(A).
		\end{tikzcd} By \cite[Theorem 6.4]{rickard_derivedmorita}, $\Phi_P$ is therefore an equivalence \begin{tikzcd}[cramped, sep=small]
			\Phi_P: D^b(A) \arrow[r,"\sim"] & D^b(A),
		\end{tikzcd} and we are done.
	\end{proof}
	
	The final step is to show that $\Phi_P$ is a perverse equivalence with the expected perversity. 
	
	\begin{prop}\label{prop:Phiperverse}
		The generalised periodic twist \begin{tikzcd}[cramped,sep=small]\Phi_P: D^b(A) \arrow[r,"\sim"] & D^b(A)\end{tikzcd} is a two-step self-perverse equivalence relative to the filtration $0 \subset_0 \mathcal{A}_1 \subset_n A\modcat$.
	\end{prop}
	
	\begin{proof}
		We appeal to Proposition \ref{prop:perverseproj}. Again, assume that $A$ is basic and let $J \subset I$ be such that $P \cong \bigoplus_{i \in I\setminus J}$ and $Q \cong \bigoplus_{j \in J}P_j$. 
		
		Recall that we have an equivalence \begin{tikzcd}[cramped,sep=small]
			\Hom_A(P,-): P\cadd \arrow[r,"\sim"] & E\proj
		\end{tikzcd} of additive categories Then there is a set of orthogonal idempotents $\{e_i\}_{i \in I \setminus J}$ such that $Pe_i \cong P_i$ as right $E$-modules. The automorphism $\sigma^{-1}$ of $E$ permutes this set; we write $P_{\sigma^{-1}(i)}$ for the summand of $P$ corresponding to $P\sigma^{-1}(e_i)$. From the proof of Proposition \ref{prop:psionP} we have that $\Phi_P(P) = X \ltens_A P \cong P \ltens_E E_{\sigma^{-1}}[n]$. For $i \in J \setminus I$, we have \[\Phi_P(P_i) \cong X \ltens_A Pe_i \cong P \ltens_E E_{\sigma^{-1}}e_i[n] \cong P \ltens_E E\sigma^{-1}(e_i)[n] \cong P_{\sigma^{-1}(i)}[n].\]		
		
		Next, for $P_j$ a direct summand of $Q$, as in the proof of Proposition \ref{prop:Phiequiv} we have a triangle
		\begin{center}
			\begin{tikzcd}[sep=small]
				P \ltens_E Y \ltens_E P^\vee \ltens_A P_j \arrow[r] & A \ltens_A P_j \arrow[r] & X \ltens_A P_j \arrow[r, rightsquigarrow] & {}
			\end{tikzcd}
		\end{center}
		and a similar argument informs us that $\Phi_P(P_j) \cong X \ltens_A P_j$ is isomorphic in $D^b(A)$ to a complex with $P_j$ in degree 0, and all other terms contained in $\langle P \rangle$. 
		
		With $\mathcal{P}$ the set of projective indecomposable $A$-modules and $\mathcal{P}_1$ the subset of direct summands of $P$, the result is clear from Proposition \ref{prop:perverseproj}.
	\end{proof}
	
	Combining Proposition \ref{prop:Phiequiv} with Proposition \ref{prop:Phiperverse} completes the proof of Theorem \ref{thm:periodictoperverse}, which together with Theorem \ref{thm:perversetoperiodic} completes the proof of Theorem \ref{thm:main}.

	\subsection{Cycle of equivalences}\label{subsec:cycle}
	
	Grant's Theorem \ref{thm:grantperiodic} produces a rather satisfying cycle of derived equivalences, \cite[Theorem 5.11]{grant_periodic}. This carries over to the generalised setting, with a slight adaptation. 
	
	Assume $A$ is basic. Again by Proposition \ref{prop:Aisbasic1} and \ref{prop:Aisbasic2}, this restriction incurs no loss of generality. Let $P$ and $Q$ be projective $A$-modules such that $A \cong P \oplus Q$ as $A$-modules. Let $J \subset I$ be the subset such that $P \cong \bigoplus_{i \in I \setminus J}P_i$ and $Q \cong \bigoplus_{j \in J}P_j$.
	
	Let \begin{tikzcd}[cramped, sep=small]
		F^{(0)}_J: D^b(A) \arrow[r,"\sim"] & D^b(A^{(1)})
	\end{tikzcd} be the elementary perverse equivalence for $A$ at $J$. This is induced by a combinatorial tilting complex $T=\bigoplus_{i \in I}T_i$. For each $i \in I$, let $P_i^{(1)} = F^{(0)}_J(T_i)$. Then the $P_i^{(1)}$ form a complete set of projective indecomposable $A^{(1)}$-modules up to isomorphism. If $P^{(1)} = \bigoplus_{i \in I \setminus J}P_i^{(1)}$ and $\mathcal{A}^{(1)}_1$ is the Serre subcategory of $A^{(1)}\modcat$ prime to $P^{(1)}$, then $F_J^{(0)}$ is perverse relative to the filtrations
	\[
	0 \subset_0 \mathcal{A}_1 \subset_{-1} A\modcat,
	\]
	\[
	0 \subset_0 \mathcal{A}^{(1)}_1 \subset_{-1} A^{(1)}\modcat.
	\]
	
	Iterating this construction, for each $i$, with $A^{(0)} = A$, let \begin{tikzcd}[cramped, sep=small]
		F^{(i)}_J: D^b(A^{(i)}) \arrow[r,"\sim"] & D^b(A^{(i+1)})
	\end{tikzcd} be the elementary perverse equivalence for $A^{(i)}$ at $J$. Set $F= F^{(n-1)}_J \circ \ldots \circ F^{(0)}_J$, so that \begin{tikzcd}[cramped, sep=small]
		F: D^b(A) \arrow[r,"\sim"] & D^b(A^{(n)})
	\end{tikzcd} is the $n$th iterated combinatorial tilt at $J$. Let $P^{(n)}_i$ be the projective indecomposable $A^{(n)}$-module obtained as the image of the $i$th summand of iterative combinatorial tilting complexes. Set $P^{(n)} = \bigoplus_{i \in I \setminus J} P_i^{(n)}$, and $\mathcal{A}^{(n)}_1$ the Serre subcategory of $A^{(n)}\modcat$ prime to $P^{(n)}$. Then the equivalence $F$ is perverse relative to the filtrations
	\[
	0 \subset_0 \mathcal{A}_1 \subset_{-n} A\modcat,
	\]
	\[
	0 \subset_0 \mathcal{A}^{(n)}_1 \subset_{-n} A^{(n)}\modcat.
	\]
	
	Now, set $E = \opcat{\End_A(P)}$ and $M=\Hom_A(P,Q)$. Suppose for some automorphism $\sigma$ of $E$ and some $\alpha \in \Ext^n_{E\otimes_k\opcat{E}}(E,{}_{\sigma\!}E)$ that $M$ is strongly $\sigma$-periodic and $M^\vee$ is strongly $\sigma^{-1}$-periodic of period $n$ relative to $\alpha$. Then by Theorem \ref{thm:periodictoperverse}, the generalised periodic twist \begin{tikzcd}[cramped, sep=small]
		\Phi_P: D^b(A) \arrow[r,"\sim"] & D^b(A) 
	\end{tikzcd} exists and is an equivalence. Moreover, $\Phi_P$ is self-perverse relative to $0 \subset_0 \mathcal{A}_1 \subset_n A\modcat$.
	
	By Lemma \ref{lem:perversecomp}, \begin{tikzcd}[cramped,sep=small]
		G=F \circ \Phi_P: D^b(A) \arrow[r,"\sim"] & D^b(A^{(n)}) 
	\end{tikzcd} is a perverse equivalence with perversity function identically zero. Thus, this induces a Morita equivalence \begin{tikzcd}[cramped, sep=small]
		G: A\modcat \arrow[r,"\sim"] & A^{(n)}\modcat.
	\end{tikzcd} By standard Morita theory, $G(A)$ is a progenerator of $A^{(n)}$, but since $A$ is basic by assumption, and $A^{(n)}$ is basic by construction, taking opposites of endomorphism rings produces an isomorphism, $A \cong A^{(n)}$. Identifying $A$ and $A^{(n)}$ via this isomorphism, we have a commutative diagram
	\begin{center}
		\begin{tikzcd}
			D^b(A) \arrow[r,"F^{-1}"] & D^b(A) \\
			D^b(A) \arrow[u,dashed,"G"] \arrow[ur,"\Phi_P"]
		\end{tikzcd}
	\end{center}
	in which all the arrows are equivalences, and the two functors $F^{-1}$ and $\Phi_P$ are naturally isomorphic. We have therefore shown the following.
	
	\begin{thm}\label{thm:generalisedcircleofequivs}
		The generalised periodic twist \begin{tikzcd}[cramped,sep=small]\Phi_P:D^b(A) \arrow[r,"\sim"] & D^b(A)\end{tikzcd} at $P$ coincides with the inverse $F^{-1}$ of the $n$th iterated combinatorial tilt $F$ at $J$. That is, for every $V \in D^b(A)$, $\Phi_P(V) = X \ltens_A V \cong F^{-1}(V)$.
	\end{thm}
	
	Thus, as in Grant's case, we obtain a cycle of derived equivalences
	\begin{center}
		\begin{tikzcd}[column sep=small]
			& D^b(A) \arrow[dr,"F_J"] \\
			D^b(A^{(n-1)}) \arrow[ur,"F_J"] & & D^b(A^{(1)}) \arrow[d,"F_J"] \\
			D^b(A^{(n-2)}) \arrow[u,"F_J"] & & D^b(A^{(2)}) \arrow[dl] \\
			& \ldots \arrow[ul]
		\end{tikzcd}
	\end{center}
	such that the complete cycle, starting and ending at $D^b(A)$, agrees with the inverse of the generalised periodic twist $\Phi_P$. As in Grant's case, we obtain for free a two-step self-perverse equivalence \begin{tikzcd}[cramped, sep=small]
		D^b(A^{(i)}) \arrow[r,"\sim"] & D^b(A^{(i)})
	\end{tikzcd} for every $i$, agreeing with the inverse of the generalised periodic twist $\Phi_{P^{(i)}}$.

	\section{Application to the Symmetric Groups}\label{sec:symmgroups}
	
	A perhaps surprising application of the results of the previous section is in the setting of the symmetric groups. The following examples are possibly an inadvertent consequence of working in small characteristic, however they are suggestive that blocks of symmetric groups, or more generally of Hecke algebras, could prove fertile ground for generating interesting periodic autoequivalences.
	
	Standard background on symmetric group representation theory can be found in \cite{james_Sn}, \cite{james-kerber_Sn}. Recall \cite{nakayama_conjecture}, \cite{brauer_nakayama} that a block $B_{\rho,w}$ of the group algebra $k\mathfrak{S}_n$ in characteristic $p$ is determined by a $p$-core partition $\rho$ and a $p$-weight $w \in \mathbb{Z}_{\ge 0}$.

	\subsection{An autoequivalence of a block of $\mathfrak{S}_6$ in characteristic 3}\label{subsec:[2:0]}

	Let $k$ be an algebraically closed field of characteristic 3, and consider the group algebra $k\mathfrak{S}_6$.  Let $A$ be the basic algebra of the block $B_{\emptyset,2}$, corresponding to the 3-core partition $\emptyset$ and of 3-weight 2. Following \cite[Theorem 7.1]{erdmann-martin_quiver}, \cite[Example 4.4]{okuyama_method}, $A$ is isomorphic to the path $k$-algebra of the quiver
	
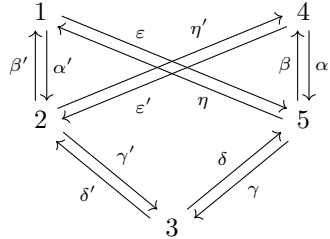
\begin{center}
		\begin{tikzcd}[sep=large]
			1 \arrow[d,shift left=0.5ex, "\alpha'"] \arrow[drr,shift left=0.5ex, "\ve" pos=0.3]  & &  4 \arrow[d,shift left=0.5ex, "\alpha"] \arrow[dll,shift left=0.5ex, "\ve'" pos=0.7] \\
			2 \arrow[dr,shift left=0.5ex, "\gamma'"] \arrow[u,shift left=0.5ex, "\beta'"] \arrow[urr,shift left=0.5ex, "\eta'" pos=0.7] & & 5 \arrow[dl,shift left=0.5ex, "\gamma"] \arrow[u,shift left=0.5ex, "\beta"] \arrow[ull,shift left=0.5ex, "\eta" pos=0.3]  \\
			& 3 \arrow[ul,shift left=0.5ex, "\delta'"] \arrow[ur,shift left=0.5ex, "\delta"]
		\end{tikzcd}
	\end{center}
	modulo the admissible ideal generated by the following list of relations:
	\begin{itemize}
		\item $\alpha\eta' = \ve\beta' = \delta\gamma'$, $\alpha'\eta = \ve'\beta = \delta'\gamma$;
		\item $\beta\ve = \eta'\alpha'$, $\beta'\ve' = \eta\alpha$, $\gamma\delta + \gamma'\delta' = 0$;
		\item $\gamma'\alpha' = \gamma\ve$, $\gamma\alpha = \gamma'\ve'$, $\beta'\delta' = \eta\delta$, $\beta\delta = \eta'\delta'$;
		\item $\delta'\gamma' = \alpha'\beta' + \ve'\eta'$, $\delta\gamma = \alpha\beta + \ve\eta$;
		\item all paths of length four starting and ending at distinct vertices are 0.
	\end{itemize}
	
	Let $P=P_1 \oplus P_2 \oplus P_4 \oplus P_5$, $Q=P_3$ and $E = \opcat{\End_A(P)}$. Then $E$ is isomorphic to the path $k$-algebra of the quiver
	
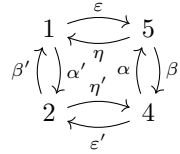
\begin{center}
		\begin{tikzcd}
			1 \arrow[r, bend left =20, "\ve"] \arrow[d, bend left =20,"\alpha'"] & 5 \arrow[l, bend left =20, "\eta"] \arrow[d, bend left =20, "\beta"] \\
			2 \arrow[u, bend left =20, "\beta'"] \arrow[r, bend left =20, "\eta'"] & 4 \arrow[l, bend left =20, "\ve'"] \arrow[u, bend left =20, "\alpha"]
		\end{tikzcd}
	\end{center}
	modulo the admissible ideal generated by the set of relations \[\{ \ve\eta\ve, \; \eta\ve\eta, \; \ve'\eta'\ve', \; \eta'\ve'\eta', \; \alpha\eta' - \ve\beta', \; \alpha'\eta - \ve'\beta, \; \beta\ve - \eta'\alpha', \; \beta'\ve' - \eta\alpha \}.\]
	
	Let $I=\{1,2,4,5\}$. The projective indecomposable $E$-modules, $\overbar{P}_i$ for $i \in I$, have coinciding Loewy and socle series
	\[
	\begin{matrix}
		1 \\
		2 \;\; 5 \\
		1 \;\; 4 \;\; 1 \\
		5 \;\; 2 \\
		1
	\end{matrix} \; ,
	\qquad
	\begin{matrix}
		2 \\
		1 \;\; 4 \\
		2 \;\; 5 \;\; 2 \\
		4 \;\; 1 \\
		2
	\end{matrix} \; ,
	\qquad
	\begin{matrix}
		4 \\
		2 \;\; 5 \\
		4 \;\; 1 \;\; 4 \\
		5 \;\; 2 \\
		4
	\end{matrix} \; ,
	\qquad
	\begin{matrix}
		5 \\
		1 \;\; 4 \\
		5 \;\; 2 \;\; 5 \\
		4 \;\; 1 \\
		5
	\end{matrix} \;.
	\]
	There is an automorphism $\sigma$ of $E$, induced by the graph automorphism of the quiver given by reflecting through the horizontal line of symmetry. The automorphism $\sigma$ acts on $I$ as the permutation $(1,2)(4,5)$.
	
	Let $M=\Hom_A(P,Q)$. Then as an $E$-module, $M$ has coinciding Loewy and socle series
	\[
	M = \; \begin{matrix}
		2 \;\; 5 \\
		1 \;\; 4 \\
		2 \;\; 5
	\end{matrix}
	\]
	and a truncated projective resolution
	\begin{center}
		\begin{tikzcd}[sep=small]
			\overbar{P}_1 \oplus \overbar{P}_4 \arrow[r] & \overbar{P}_2 \oplus \overbar{P}_5 \arrow[r]  & M.
		\end{tikzcd}
	\end{center}
	One can calculate that $\Omega_E^2(M) \cong {}_{\sigma\!}M$, so that $M$ is $\sigma$-periodic of period 2. We claim that $M$ is strongly $\sigma$-periodic of period 2, and, noting that $\sigma^{-1} = \sigma$, so is $M^\vee$.
	
	Let $B$ be the subalgebra of $E$ generated by the horizontal arrows: the path $k$-algebra of the quiver
	
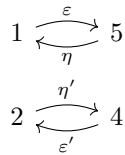
\begin{center}
		\begin{tikzcd}
			1 \arrow[r, bend left =20, "\ve"] & 5 \arrow[l, bend left =20, "\eta"]  \\
			2 \arrow[r, bend left =20, "\eta'"] & 4 \arrow[l, bend left =20, "\ve'"] 
		\end{tikzcd}
	\end{center}
	subject to the relations $\{\ve\eta\ve, \eta\ve\eta, \ve'\eta'\ve', \eta'\ve'\eta'\}$. Then $B \cong A_{2,1} \times A_{2,1}$, with $A_{2,1}$ the Brauer tree algebra of a star on 2 edges with exceptional multiplicity $m=1$. The automorphism $\sigma$ restricted to $B$ swaps the two direct factors. The projective indecomposable $B$-modules, say $Q_i$ for $i \in I$, have Loewy and socle series
	\[
	\begin{matrix}
		1 \\
		5 \\
		1
	\end{matrix} \; ,
	\qquad
	\begin{matrix}
		2 \\
		4 \\
		2 
	\end{matrix} \; ,
	\qquad
	\begin{matrix}
		4 \\
		2 \\
		4
	\end{matrix} \; ,
	\qquad
	\begin{matrix}
		5 \\
		1 \\
		5
	\end{matrix} \;.
	\]
	It is not too difficult to see that $E$ and $M$ are projective as $B$-modules. There are relatively $B$-projective $E$-modules $U_1$, $U_2$, $U_4$ and $U_4$, with Loewy and socle series
	\[
	U_1 = \; \begin{matrix}
		1 \\
		2 \\
		1
	\end{matrix} \; , \; \; U_2 = \; \begin{matrix}
		2 \\
		1 \\
		2
	\end{matrix} \; , \; \; U_4 = \; \begin{matrix}
		4 \\
		5 \\
		4
	\end{matrix} \; , \; \; U_5 = \; \begin{matrix}
		5 \\
		4 \\
		5
	\end{matrix} \; .
	\]
	
	We claim that there is an exact sequence of $E$-$E$-bimodules of the form
	\begin{center}
		\begin{tikzcd}[sep=small]
			0 \arrow[r] & {}_{\sigma\!}E \arrow[r,"d_2"] & {}_{\sigma\!}E \otimes_B E \arrow[r,"d_1"] & E \otimes_B E \arrow[r,"d_0"] & E \arrow[r] & 0.
		\end{tikzcd}
	\end{center}

	Applying the functors $- \ltens_E S_i$, where $S_i$ are the simple $E$-modules, gives complexes of the form
	\begin{center}
		\begin{tikzcd}[sep=small]
			0 \arrow[r] & S_2 \arrow[r] & U_2 \arrow[r] & U_1 \arrow[r] & S_1 \arrow[r] & 0,
		\end{tikzcd}
	\end{center}
	\begin{center}
		\begin{tikzcd}[sep=small]
			0 \arrow[r] & S_1 \arrow[r] & U_1 \arrow[r] & U_2 \arrow[r] & S_2 \arrow[r] & 0,
		\end{tikzcd}
	\end{center}
	\begin{center}
		\begin{tikzcd}[sep=small]
			0 \arrow[r] & S_5 \arrow[r] & U_5 \arrow[r] & U_4 \arrow[r] & S_4 \arrow[r] & 0,
		\end{tikzcd}
	\end{center}
	\begin{center}
		\begin{tikzcd}[sep=small]
			0 \arrow[r] & S_4  \arrow[r] & U_4 \arrow[r] & U_5 \arrow[r] & S_5 \arrow[r] & 0.
		\end{tikzcd}
	\end{center}
	
	Thus, the algebra $E$ is $\sigma$-periodic of period 2, relative to the subalgebra $B$.	We thus have a triangle
	\begin{center}
		\begin{tikzcd}[sep=small]
			Y \arrow[r] & E \arrow[r,"\alpha"] & {}_{\sigma\!}E[2] \arrow[r,rightsquigarrow] & {}
		\end{tikzcd}
	\end{center}
	in $\bbimder{E}{E}$, defining $\alpha \in \Ext^2_{E \otimes_k \opcat{E}}(E,{}_{\sigma\!}E)$. By Proposition \ref{prop:grantperiodicmodule}, both $M$ and $M^\vee$ are strongly $\sigma$-periodic of period 2, relative to $\alpha$. The resulting generalised periodic twist, given by Theorem \ref{thm:periodictoperverse}, is the Grantian periodic twist at $P$ relative to $B$, \begin{tikzcd}[cramped, sep=small]
		\Phi_P:D^b(A) \arrow[r,"\sim"] & D^b(A),
	\end{tikzcd} perverse relative to the filtration $\emptyset \subset_0 \{3\} \subset_2 I$.
	
	We comment that a perverse autoequivalence of $A$ of this form is already known to exist. The block $B_{\emptyset,2}$ admits two autoequivalences arising from Scopes $[2:0]$ pairs \cite[Definition 2.1]{scopes_weight2}, say $\Phi_1$ and $\Phi_2$. By \cite[Theorem 7.2]{chuang-rouquier_sl2cat}, \cite[Proposition 8.4]{chuang-rouquier_perverse17}, both of these functors are self-perverse equivalences of width 1. By \cite[Theorem 2.10]{cautis-kamnitzer_braid}, the braid relation $\Phi_1\Phi_2\Phi_1 \cong \Phi_2\Phi_1\Phi_2$ holds for these two equivalences. Further, by a result of Halacheva, Licata, Losev and Yacobi \cite[Theorem 6.8, Remark 6.9]{halacheva-etal_braid}, the braid
	\begin{center}
		\begin{tikzcd}[sep=small]
			\Phi_1\Phi_2\Phi_1: D^b(B_{\emptyset,2}) \arrow[r,"\sim"] & D^b(B_{\emptyset,2})
		\end{tikzcd}
	\end{center}
	is a self-perverse equivalence relative to the \emph{isotypic filtration} $
	\emptyset \subset_0 \{3\} \subset_2 I$. The uniqueness of perverse equivalences tells us that $\Phi_P \cong \Phi_1\Phi_2\Phi_1$. 
	
	It seems likely that the restriction to the subalgebra $B \cong A_{2,1} \times A_{2,1}$ in $\Phi_P$ is in some way masking the restriction to weight 1 blocks of $k\mathfrak{S}_5$ involved in the combinatorial description of the $[2:0]$ pairs $\Phi_1$, $\Phi_2$. Regrettably, we are unable to say anything more precise about the relationship between these two formulations.

	\subsection{An autoequivalence of a block of $\mathfrak{S}_8$ in characteristic 3}\label{subsec:exotic}

	Let $k$ be an algebraically closed field of characteristic 3. Let $A$ be the basic algebra of the block $B_{(2),2}$ of $k\mathfrak{S}_8$ with 3-core $(2)$ and of 3-weight 2. Following \cite[Example 4.3]{okuyama_method}, $A$ is isomorphic to the path $k$-algebra of the quiver
	
	
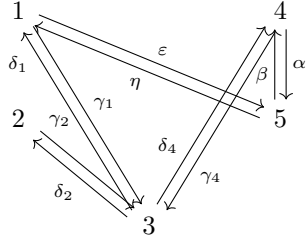
\begin{center}
		\begin{tikzcd}[sep = large]
			1 \arrow[ddr, shift left=0.5ex, "\gamma_1"] \arrow[drr, shift left=0.5ex, "\ve"]  & &  4 \arrow[ddl, shift left=0.5ex, "\gamma_4" pos=.75] \arrow[d, shift left=0.5ex, "\alpha"] \\
			2 \arrow[dr, shift left=0.5ex, "\gamma_2" pos=.025] & & 5 \arrow[u, shift left=0.5ex, "\beta", pos=0.35] \arrow[ull, shift left=0.5ex, "\eta"] \\
			& 3 \arrow[ul, shift left=0.5ex, "\delta_2"] \arrow[uul, shift left=0.5ex, "\delta_1" pos=.9] \arrow[uur, shift left=0.5ex, "\delta_4" pos=.25]
		\end{tikzcd}
	\end{center}
	
	modulo the admissible ideal generated by the following list of relations:
	\begin{itemize}
		\item $\beta\ve = \delta_4\gamma_1$, $\eta\alpha = \delta_1\gamma_4$, $\ve\delta_1 = \alpha\delta_4$, $\gamma_1\eta = \gamma_4\beta$;
		\item $\alpha\beta = \ve\eta$, $\gamma_1\delta_1 + \gamma_4\delta_4 = \gamma_2\delta_2$;
		\item $\alpha\delta_4\gamma_2 = 0$, $\delta_2\gamma_4\beta = 0$;
		\item $\gamma_1\delta_1\gamma_1 = 0$, $\delta_1\gamma_1\delta_1 = 0$, $\gamma_4\delta_4\gamma_4 = 0$, $\delta_4\gamma_4\delta_4 = 0$;
		\item $\delta_2\gamma_1\delta_1 = \delta_2\gamma_4\delta_4$, $\gamma_1\delta_1\gamma_2 = \gamma_4\delta_4\gamma_2$;
		\item all paths of length four starting and ending at distinct vertices are 0.
	\end{itemize}
	
	Let $P=P_2 \oplus P_3 \oplus P_4 \oplus P_5$, $Q=P_1$, $E=\opcat{\End_A(P)}$ and $M=\Hom_A(P,Q)$. Then $E$ is isomorphic to the path algebra of the quiver
	
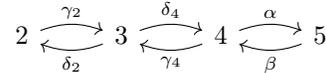
\begin{center}
		\begin{tikzcd}
			2 \arrow[r, bend left=20, "\gamma_2"] & 3 \arrow[l, bend left=20, "\delta_2"] \arrow[r, bend left=20, "\delta_4"] & 4 \arrow[l, bend left=20, "\gamma_4"] \arrow[r, bend left=20, "\alpha"] & 5 \arrow[l, bend left=20, "\beta"]
		\end{tikzcd}
	\end{center}
	modulo the admissible ideal generated by the following list of relations:
	\begin{itemize}
		\item $\alpha\delta_4\gamma_2 = 0 = \delta_2\gamma_4\beta$;
		\item $\delta_4\gamma_4\delta_4 = 0 = \gamma_4\delta_4\gamma_4$;
		\item $\gamma_2\delta_2\gamma_2 + \gamma_4\delta_4\gamma_2 = 0 = \delta_2\gamma_2\delta_2 + \delta_2\gamma_4\delta_4$;
		\item $\delta_4\gamma_2\delta_2 + \beta\alpha\delta_4 = 0 = \gamma_2\delta_2\gamma_4 + \gamma_4\beta\alpha$;
		\item $\alpha\delta_4\gamma_4 + \alpha\beta\alpha = 0 = \delta_4\gamma_4\beta + \beta\alpha\beta$;
		\item all paths of length four between distinct vertices are zero.
	\end{itemize}
	
	Let $I = \{2,3,4,5\}$. The projective indecomposable $E$-modules, say $\overbar{P}_i$ for $i \in I$, have coinciding Loewy and socle series
	\[
	\begin{matrix}
		2 \\
		3 \\
		2 \;\; 4 \\
		3 \\
		2
	\end{matrix}
	\qquad 
	\begin{matrix}
		3 \\
		2 \;\; 4 \\
		3 \;\; 5 \;\; 3 \\
		4 \;\; 2 \\
		3
	\end{matrix}
	\qquad 
	\begin{matrix}
		4 \\
		3 \;\; 5 \\
		4 \;\; 2 \;\; 4 \\
		5 \;\; 3 \\
		4
	\end{matrix}
	\qquad 
	\begin{matrix}
		5 \\
		4 \\
		3 \;\; 5 \\
		4 \\
		5
	\end{matrix} \; .
	\]
	
	There is an automorphism $\sigma$ of $E$, induced by the graph automorphism of the above quiver given by rotating the quiver $180^\circ$ about the centre. Then $\sigma$ acts on $I$ as the permutation $(2,5)(3,4)$.
	
	The $E$-module $M$ has Loewy and socle series given by
	\[
	M \; = \; \begin{matrix}
		3 \;\; 5 \\
		2 \;\; 4 \\
		3 \;\; 5
	\end{matrix} \; 
	\]
	and a truncated projective resolution
	\begin{center}
		\begin{tikzcd}[sep=small]
			\overbar{P}_2 \oplus \overbar{P}_4 \arrow[r] & \overbar{P}_4 \oplus \overbar{P}_3 \arrow[r] & \overbar{P}_3 \oplus \overbar{P}_5 \arrow[r] & M.
		\end{tikzcd}
	\end{center}
	
	One can then calculate that $\Omega_E^3(M) \cong {}_{\sigma\!}M$, so that $M$ is a $\sigma$-periodic $E$-module of period 3. We claim that there is an $\alpha \in \Ext_{E\otimes_k\opcat{E}}^3(E,{}_{\sigma\!}E)$ such that $M$ is strongly $\sigma$-periodic relative to $\alpha$. 
	
	We describe a construction, Grantian in nature, with a complex of $E$-$E$-bimodules constructed from terms projective relative to some subalgebras of $E$. We first identify these subalgebras.
	
	For $i \in I$, let $e_i$ be the primitive idempotent of $E$ such that $\overbar{P}_i = Ee_i$. Let $B$ be the subalgebra of $E$ generated by the idempotents $e=e_2+e_4$ and $f=e_3 + e_5$ and the arrows $\zeta = \gamma_2 + \gamma_4 + \alpha$ and $\xi = \delta_2 + \delta_4 + \beta$. Then $B$ is the path $k$-algebra of the quiver
	\begin{center}
		\begin{tikzcd}
			e \arrow[r,bend left=20,"\zeta"] & f \arrow[l,bend left=20,"\xi"] 
		\end{tikzcd}
	\end{center}
	modulo the admissible ideal generated by the set of relations $\{\zeta\xi\zeta, \xi\zeta\xi\}$. We comment that $B \cong A_{2,1}$ as $k$-algebras, where $A_{2,1}$ is the Brauer tree algebra of a star on two edges with exceptional multiplicity $m=1$.
	
	Next, let $C$ be the subalgebra of $E$ generated by the idempotents $e_2, e_3, e_4, e_5$ and the arrows $\gamma_4, \delta_4$. Then $C$ is the path algebra of the quiver
	\begin{center}
		\begin{tikzcd}
			2 & 3  \arrow[r, bend left=20, "\delta_4"] & 4 \arrow[l, bend left=20, "\gamma_4"]  & 5
		\end{tikzcd}
	\end{center}
	modulo the admissible ideal generated by the set of relations $\{\gamma_4\delta_4\gamma_4, \delta_4\gamma_4\delta_4\}$. Then $C \cong k \times A_{2,1} \times k$ as $k$-algebras. Note that $C$ is not an indecomposable algebra.
	
	Finally, let $D$ be the subalgebra generated by the four idempotents $e_2,e_3,e_4,e_5$. Then $D \cong k \times k \times k \times k$ as $k$-algebras. Again, $D$ is not an indecomposable algebra.
	
	Consider the sequence of $E$-$E$-bimodules
	\begin{center}
		\begin{tikzcd}[sep=small]
			0 \arrow[r] & {}_{\sigma\!}E \arrow[r,"d_3"] &  {}_{\sigma\!}E \otimes_C E \arrow[r,"d_2"] & E \otimes_{{}_{\tau\!}D} E \arrow[r,"d_1"] & E \otimes_B E \arrow[r,"d_0"] & E \arrow[r] & 0
		\end{tikzcd}
	\end{center}
	with differentials defined below, and where $E \otimes_{{}_{\tau\!}D} E$ denotes the $E$-$E$-bimodule $E \otimes_D {}_{\tau\!}D \otimes_D E$, where $\tau$ is the automorphism of $D$ acting as the permutation $(2,4)(3,5)$ on labels of simple $D$-modules. We comment that $E \otimes_{{}_{\tau\!}D} E = \bigoplus_{i \in I} \overbar{P}_{i\tau(i)}$, where $\overbar{P}_{i\tau(i)} = \overbar{P}_i \otimes_k \overbar{P}_{\tau(i)}^\vee$, so this term is projective as an $E$-$E$-bimodule.
	
	The map $d_0:E\otimes_B E \to E$ is the multiplication map, $d_0(x \otimes y) = xy$. The map $d_1: E \otimes_{{}_{\tau\!}D} E \to E \otimes_B E$ is given by $d_1(e_i \otimes e_{\tau(i)}) = e_i \otimes e_{\tau(i)}$. These elements generate $E \otimes_{{}_{\tau\!}D} E$, so this completely defines $d_1$. Further, $d_1$ is well-defined, as we need only consider the action of idempotents in $E$ on either side. Then $d_0 \circ d_1 = 0$, since
	\[
	d_0(d_1(e_i \otimes e_{\tau(i)})) = d_0(e_i \otimes e_{\tau(i)}) = e_ie_{\tau(i)} = 0.
	\]
	
	Next, we define $d_2: {}_{\sigma\!}E \otimes_C E \to E \otimes_{{}_{\tau\!}D} E$ as follows. We set
	\begin{align*}
		d_2(e_2 \otimes e_2) = &\,\alpha \otimes e_2 - e_5 \otimes \gamma_2, \\
		d_2(e_3 \otimes e_3) = &\,\delta_4\gamma_4 \otimes \delta_2 + \beta \otimes \gamma_4\delta_4 + \delta_4 \otimes \alpha\delta_4 + \delta_4\gamma_2 \otimes \delta_4  \\
		&\,- \beta\alpha\beta \otimes e_3 - e_4 \otimes \delta_2\gamma_2\delta_2, \\
		d_2(e_4 \otimes e_4) = &\,\gamma_4\delta_4 \otimes \alpha + \gamma_2 \otimes \delta_4\gamma_4 + \gamma_4 \otimes \delta_2\gamma_4 + \gamma_4\beta \otimes \gamma_4  \\
		&\,- \gamma_2\delta_2\gamma_2 \otimes e_4 - e_3 \otimes \alpha\beta\alpha., \\
		d_2(e_5 \otimes e_5) = &\,\delta_2 \otimes e_5 - e_2 \otimes \beta.
	\end{align*}
	Then one can show that $d_2(\delta_4 \otimes e_3) = d_2(e_4 \otimes \delta_4)$ and $d_2(e_3 \otimes \gamma_4) = d_2(\gamma_4 \otimes e_4)$, so that $d_2$ is well-defined. We have 
	\begin{align*}
		d_1(d_2(e_2 \otimes e_2)) &= d_1(e_5\alpha e_4 \otimes e_2 - e_5 \otimes e_3 \gamma_2 e_2) \\
		&= d_1(e_5\zeta e_4 \otimes e_2 - e_5 \otimes e_3\zeta e_2) \\
		&= d_1(e_5(\zeta \otimes 1_E - 1_E \otimes \zeta)e_2) \\
		&= e_5(\zeta \otimes 1_E - 1_E \otimes \zeta)e_2 \\
		&= 0, 
	\end{align*}
	while
	\begin{align*}
		d_1(d_2(e_5 \otimes e_5)) &= d_1(e_2\delta_2 e_3 \otimes e_5 - e_2 \otimes e_4 \beta e_5) \\
		&= d_1(e_2\xi e_3 \otimes e_5 - e_2 \otimes e_4 \xi e_5) \\
		&= d_1(e_2(\xi \otimes 1_E - 1_E \otimes \xi)e_5) \\
		&= e_2(\xi \otimes 1_E - 1_E \otimes \xi)e_5 \\
		&= 0.
	\end{align*}
	Next, we have that
	\[
	d_1(d_2(e_3 \otimes e_3)) = e_4(\xi\zeta \otimes \xi + \xi \otimes \zeta\xi - 1_E \otimes \xi\zeta\xi - \xi\zeta\xi \otimes 1_E)e_3 = 0, 
	\]
	while
	\[
	d_1(d_2(e_4 \otimes e_4)) = e_3(\zeta\xi \otimes \zeta + \zeta \otimes \xi\zeta - 1_E \otimes \zeta\xi\zeta - \zeta\xi\zeta \otimes 1_E)e_4 = 0,
	\]
	since $\zeta\xi\zeta = 0 = \xi\zeta\xi$ in $B$. Thus, $d_2 \circ d_1 = 0$.
	
	Finally, we define $d_3: {}_{\sigma\!}E \to {}_{\sigma\!}E \otimes_C E$ by
	\[
	d_3(1_E) = y_2 + y_{34} + y_5,
	\]
	where 
	\begin{align*}
		y_2 = & \; e_2 \otimes \delta_2\gamma_2\delta_2\gamma_2 + \delta_2\gamma_2 \otimes \delta_2\gamma_2 + \delta_2\gamma_2\delta_2\gamma_2 \otimes e_2 \\
		&+\gamma_2\delta_2\gamma_2 \otimes \delta_2 + \gamma_2 \otimes \delta_2\gamma_2\delta_2 - \delta_4\gamma_2 \otimes \delta_2\gamma_4 \; \in Ee_2 \otimes_C e_2E, \\
		y_{34} = & \; \delta_2 \otimes \gamma_2 - \alpha \otimes \beta + \gamma_2\delta_2 \otimes e_3 - \beta\alpha \otimes e_4 + e_3 \otimes \gamma_2\delta_2  \\
		&- e_4 \otimes\beta\alpha - \gamma_4 \otimes \delta_4 + \delta_4 \otimes \gamma_4 \; \in E(e_3 + e_4) \otimes_C (e_3 + e_4)E, \\
		y_5 = & \; e_5 \otimes \alpha\beta\alpha\beta + \alpha\beta \otimes \alpha\beta + \alpha\beta\alpha\beta \otimes e_5 \\
		&+ \beta\alpha\beta \otimes \alpha + \beta \otimes \alpha\beta\alpha - \gamma_4\beta \otimes \alpha\delta_4 \; \in Ee_5 \otimes_C e_5E.
	\end{align*}
	
	Observe that $y_5 = \sigma(y_2)$. The element $d_3(1_E)$ is central in $E \otimes_C E$, so $d_3$ is a well-defined homomorphism of bimodules. A rather painstaking calculation will then show that $d_2(d_3(1_E)) = 0$ in $E \otimes_{{}_{\tau\!}D} E$, and our sequence is a complex of $E$-$E$-bimodules.
	
	To show that this complex is an exact sequence, one considers the complexes obtained by applying the functors $- \otimes_E S_i$, for $S_i$ the simple $E$-modules. There are relatively $B$-projective $E$-modules $U_{24}$ and $U_{35}$, with Loewy series
	\[
	U_{24} \; = \; \begin{matrix}
		2 \;\;\; 4 \\
		3 \\
		2 \;\;\; 4 
	\end{matrix} \quad \text{ and } \quad U_{35} \; = \; \begin{matrix}
		3 \;\;\; 5 \\
		4 \\
		3 \;\;\; 5
	\end{matrix} \; .  
	\]
	The projective $E$-modules $\overbar{P}_2$ and $\overbar{P}_5$ are relatively $C$-projective, as are the modules $V_3$ and $V_4$, with Loewy series 
	\[
	V_3 \; = \; \begin{matrix}
		3 \\
		2 \\
		3 
	\end{matrix} \quad \text{ and } \quad V_4 \; = \; \begin{matrix}
		4 \\
		5 \\
		4
	\end{matrix} \; .  
	\]
	
	Applying the functors $- \otimes_E S_i$, we obtain exact sequences
	\begin{center}
		\begin{tikzcd}[column sep=small]
			0 \arrow[r] & S_5 \arrow[r] & \overbar{P}_5 \arrow[r] & \overbar{P}_4 \arrow[r] & U_{24} \arrow[r] & S_2 \arrow[r] & 0,
		\end{tikzcd}
	\end{center}
	\begin{center}
		\begin{tikzcd}[sep=small]
			0 \arrow[r] & S_4 \arrow[r] & V_4 \arrow[r] & \overbar{P}_5 \arrow[r] & U_{35} \arrow[r] & S_3 \arrow[r] & 0,
		\end{tikzcd}
	\end{center}
	\begin{center}
		\begin{tikzcd}[sep=small]
			0 \arrow[r] & S_3 \arrow[r] & V_3 \arrow[r] & \overbar{P}_2 \arrow[r] & U_{24} \arrow[r] & S_4 \arrow[r] & 0, 
		\end{tikzcd}
	\end{center}
	\begin{center}
		\begin{tikzcd}[sep=small]
			0 \arrow[r] & S_2 \arrow[r] & \overbar{P}_2 \arrow[r] & \overbar{P}_3 \arrow[r] & U_{35} \arrow[r] & S_5 \arrow[r] & 0.
		\end{tikzcd}
	\end{center}
	Thus, the above sequence of $E$-$E$-bimodules is exact. This gives rise to a triangle
	\begin{center}
		\begin{tikzcd}[sep=small]
			Y \arrow[r] & E \arrow[r,"\alpha"] & {}_{\sigma\!}E[3] \arrow[r,rightsquigarrow] & {}
		\end{tikzcd}
	\end{center}
	in $\bbimder{E}{E}$, with $Y$ perfect in $D^b(E)$ and $D^b(\opcat{E})$ by definition. This defines an element $\alpha \in \Ext^3_{E \otimes_k \opcat{E}}(E,{}_{\sigma\!}E)$. Since $E$ is projective as a left and a right $B$-module, $C$-module and $D$-module, we have that $Y \ltens_E M$ is perfect in $D^b(E)$ and $M^\vee \ltens_E Y$ is perfect in $D^b(\opcat{E})$. Noting that $\sigma^{-1} = \sigma$, both $M$ and $M^\vee$ are thus strongly $\sigma$-periodic of period 3, relative to $\alpha$. Thus, by Theorem \ref{thm:periodictoperverse} we have a generalised periodic twist \begin{tikzcd}[cramped, sep=small]
		\Phi_P: D^b(A) \arrow[r,"\sim"] & D^b(A),
	\end{tikzcd} self-perverse relative to the filtration $\emptyset \subset_0 \{1\} \subset_3 I$.
	
	We comment that clever manipulation of a number of previously known equivalences will produce an autoequivalence of $B_{(2),2)}$ with the same perversity data as $\Phi_P$.
	
	Let \begin{tikzcd}[cramped, sep=small]
		\Phi_1:D^b(B_{(2),2}) \arrow[r,"\sim"] & D^b(B_{(3,1^2),2})
	\end{tikzcd} be the equivalence of Craven and Rouquier realising Brou\'{e}'s abelian defect group conjecture in \cite[Section 5.5.3]{craven-rouquier_perverse}, where $B_{(3,1^2),2}$ is the weight 2 Rouquier block, Morita equivalent in this case to the Brauer correspondent $A_{2,1} \wr \mathfrak{S}_2$ by \cite[Theorem 2]{chuang-kessar_badgc} and the exceptional coincidence that the Brauer tree algebras of a star and a line on two edges with exceptional multiplicity 1.
	
	Next, let \begin{tikzcd}[cramped, sep=small]
		\Phi_2:D^b(B_{(3,1^2),2}) \arrow[r,"\sim"] & D^b(B_{(3,1^2),2})
	\end{tikzcd} be the equivalence of \cite[Example 5.7]{marcus_equiv} lifting the line-to-star derived equivalence for Brauer tree algebras of \cite[Theorem 4.2]{rickard_derivedstable}.
	
	Finally, let \begin{tikzcd}[cramped, sep=small]
		\Phi_3: D^b(B_{(3,1^2),2}) \arrow[r,"\sim"] & D^b(B_{(2),2})
	\end{tikzcd} be the equivalence of \cite[Theorem 7.2]{chuang-rouquier_sl2cat} arising from the $[2:1]$ pair between $B_{(3,1^2),2}$ and $B_{(2),2}$. The equivalence 
	\begin{center}
		\begin{tikzcd}[sep=small]
			\Phi_3 \circ \Phi_2^{-1} \circ \Phi_1: D^b(B_{(2)}) \arrow[r,"\sim"] & D^b(B_{(2)})
		\end{tikzcd}
	\end{center}
	is a self-perverse equivalence, relative to the filtration $\emptyset \subset_0 \{1\} \subset_3 I$.
	
	Again, the appearance of the subalgebra $B$ and the constituent of $C$  isomorphic to the Brauer tree algebra $A_{2,1}$ is highly suggestive of the functor $\Phi_{P}$ masking some restriction and induction functors to some weight one blocks. We propose that this is related to the self-stable equivalences of the Brauer correspondent block described by Craven and Rouquier in \cite[Section 5.5]{craven-rouquier_perverse}, though it is not entirely clear precisely how.

	\emergencystretch=2em

	\printbibliography

\end{document}

\typeout{get arXiv to do 4 passes: Label(s) may have changed. Rerun}